\documentclass[a4,10pt]{article}
\usepackage{graphicx,url,amsmath,amssymb,amsthm,palatino,subfigure,siunitx,tikz,pgfplots,anysize}
\usepackage[affil-it]{authblk}
\marginsize{2cm}{2cm}{2cm}{2cm}

\newcommand{\eps}{\varepsilon}
\newcommand{\epsb}{\varepsilon_{\mathrm{b}}}
\newcommand{\mub}{\mu_{\mathrm{b}}}

\newcommand{\set}[1]{\left\{#1\right\}}

\newcommand{\p}{\partial}
\newcommand{\qand}{\quad\text{and}\quad}

\newcommand{\mrp}{\mathbf{p}}
\newcommand{\mq}{\mathbf{q}}
\newcommand{\mr}{\mathbf{r}}
\newcommand{\mB}{\mathbf{B}}

\newcommand{\mG}{\mathbb{G}}
\newcommand{\mM}{\mathbb{M}}
\newcommand{\mP}{\mathbb{P}}
\newcommand{\mQ}{\mathbb{Q}}
\newcommand{\mS}{\mathbb{S}}
\newcommand{\rd}{\mathrm{d}}
\newcommand{\vt}{\boldsymbol{\theta}}
\newcommand{\vv}{\boldsymbol{\vartheta}}

\DeclareMathOperator*{\dsm}{DSM}
\DeclareMathOperator*{\msm}{MSM}
\DeclareMathOperator*{\scat}{scat}
\DeclareMathOperator*{\area}{area}
\DeclareMathOperator*{\sinc}{sinc}

\theoremstyle{plain}
\newtheorem{Theorem}{Theorem}[section]

\theoremstyle{remark}
\newtheorem{Example}{Example}[section]
\newtheorem{Discussion}{Discussion}[section]

\begin{document}
\title{Direct Sampling Method to Retrieve Small Objects From Two-dimensional Limited-Aperture Scattered Field Data}
\author{Won-Kwang Park}
\affil{Department of Information Security, Cryptology, and Mathematics, Kookmin University, Kookmin University, Seoul, 02707, Korea. Electronic address: parkwk@kookmin.ac.kr}
\date{}
\maketitle
\begin{abstract}
In this study, we investigated the application of the direct sampling method (DSM) to identify small dielectric objects in a limited-aperture inverse scattering problem. Unlike previous studies, we consider the bistatic measurement configuration corresponding to the transmitter location and design indicator functions for both a single source and multiple sources, and we convert the unknown measurement data to a fixed nonzero constant. To explain the applicability and limitation of object detection, we demonstrate that the indicator functions can be expressed by an infinite series of Bessel functions, the material properties of the objects, the bistatic angle, and the converted constant. Based on the theoretical results, we explain how the imaging performance of the DSM is influenced by the bistatic angle and the converted constant. In addition, the results of our analyses demonstrate that a smaller bistatic angle enhances the imaging accuracy and that optimal selection of the converted constant is crucial to realize reliable object detection. The results of the numerical simulations obtained using a two-dimensional Fresnel dataset validated the theoretical findings and illustrate the effectiveness and limitations of the designed indicator functions for small objects.
\end{abstract}

\section{Introduction}
This paper addresses the inverse scattering problem of localizing a set of small objects using scattered field data collected using a limited-aperture measurement system. Inverse scattering problems include important research topics in mathematics, physics, and engineering because they play a crucial role in various applications that are relevant to human life, including medical imaging (e.g., breast cancer detection \cite{SMSEKAKO}, brain stroke diagnosis \cite{PFTYMPKE}, and thermal therapy monitoring \cite{HSM2}), nondestructive testing (e.g., damage detection in concrete structures \cite{FFK}, eddy-current testing of damaged plates \cite{HLD}, and surface crack detection \cite{FMGD}), and radar applications (e.g., human heart motion imaging \cite{BBPALH}, mine detection in a two-layered medium \cite{DEKPS}, and through-wall imaging \cite{WZLD}). Numerous studies \cite{A2,ABT,AK2,BCS,C7,CK,N3,P5} have discussed related theories and applications. Despite being an important research topic, it is extremely challenging to solve inverse scattering problems due to their inherent nonlinearity and ill-posedness. Thus, various iterative (or quantitative) and noniterative (or qualitative) inversion techniques have been investigated.

Generally, iterative techniques have been developed to reconstruct the parameter (dielectric permittivity, electric conductivity, or magnetic permeability) distribution of a bounded domain. For example, the Newton-type method was proposed for shape reconstruction of an arc-like perfectly conducting crack \cite{K}, the Gauss--Newton method was proposed for electrical impedance tomography (EIT) \cite{ASKK}, the Born iterative technique was proposed for brain stroke detection \cite{IBA}, the Levenberg--Marquardt method was proposed to reconstruct two-dimensional isotropic and anisotropic inhomogeneities \cite{LLHH}, and the Newton--Kantorovich algorithm was developed to retrieve the permittivity distribution of the human thorax and arm \cite{MJBB}. However, as confirmed by various previous studies \cite{KSY,PL4}, the iterative process must begin with a good initial guess that is close to the true solution to ensure successful application of iteration-based algorithms and avoid various critical issues, e.g., such as nonconvergence, becoming trapped in local minima, and high computational costs. Consequently, fast algorithms are required to obtain a good initial guess.

Motivated by this issue, previous studies have investigated various noniterative techniques to retrieve the existence, location, or outline shape of arbitrary shaped objects. For example, the direct sampling method (DSM) has gained increasing popularity due to its computational efficiency and robustness against noise. Thus, the DSM has been applied in various interesting problems, e.g., the localization of two-dimensional and three-dimensional scatterers in full-view \cite{IJZ1,IJZ2,KLP1,KL4,P-DSM9} and limited-view \cite{KLAHP} inverse scattering problems. In addition, the DSM has been applied in EIT \cite {CIZ}, diffusive optical tomography \cite{CILZ}, and monostatic imaging \cite{KLP3}. The DSM has also been employed for the shape reconstruction of arbitrary shaped scatterers from far-field measurement data \cite{LZ,HK}, the identification of multipolar acoustic sources \cite{BGWL} and small perfectly conducting cracks \cite{P-DSM3}, the detection of small inhomogeneities in transverse electric polarized waves \cite{AHP2,P-DSM1}, and anomaly detection from scattering parameter data in microwave imaging \cite{P-DSM2,SLP1}. Previous studies have confirmed that the DSM is a fast, stable, and effective technique in full-view and limited-view/aperture inverse scattering problems. However, the data acquisition processes in many real-world applications are limited. For example, the diagonal elements of the scattering matrix cannot be obtained because each antenna is used exclusively for signal transmission, and the remaining antennas are utilized for signal reception (refer to the literature \cite{KLKJS} for a detailed description). Following the inverse scattering problem using an experimental Fresnel dataset \cite{BS}, a receiver rotates within a limited range based on each transmitter’s direction when the transmitter is located at a fixed position to avoid interference between the transmitter and receiver. Previous studies \cite{KLPS1,P-OSM1,P-SUB19} have demonstrated that converting unmeasurable data with a zero constant guarantees good results; however, to the best of our knowledge, the theoretical implications of this approach to the DSM have not been explored. In addition, the impact of the bistatic angle on imaging performance has not been analyzed theoretically.

Thus, in this paper, we consider the application of the DSM to a real-world limited-aperture inverse scattering problem to identify a set of small objects. To this end, we design an indicator function of the DSM by converting unmeasurable scattered field data to a fixed nonzero constant. In addition, to explore the impact of the converted constant and bistatic angle, we demonstrate that the designed indicator function can be expressed by an infinite series of Bessel functions of integer order of the first kind, the material properties of the objects and the background, the converted constant, and the bistatic angle. Based on the theoretical results, it can be explained that imaging performance is strongly dependent on both the converted constant and the bistatic angle. The conversion of the zero constant ensures good imaging results; however, setting the bistatic angle to $\SI{180}{\degree}$ does not. To demonstrate this theoretical result, various numerical simulation results conducted using a Fresnel dataset are discussed and analyzed.

The remainder of this paper is organized as follows. Section \ref{sec:2} introduces the two-dimensional problem, including the configuration of the limited-aperture data measurement process and the formulation of the integral equation for the scattered field in the presence of a set of small objects. Section \ref{sec:3} presents the design of the indicator function for the DSM, the mathematical structure of the designed indicator function with a single source, and its various properties. Then, Section \ref{sec:4} extends the framework to the DSM with multiple sources, explores its mathematical structure and discusses its various properties, including the observed improvements. Section \ref{sec:5} discusses and analyzes the numerical simulation results obtained on the Fresnel dataset that support our theoretical findings. Finally, the paper is concluding in Section \ref{sec:6} with a summary of key insights and potential future research directions.

\section{Problem Setup and Scattered Field}\label{sec:2}
Here, we summarize the basic concept of the scattered field in the presence of small dielectric objects and the configuration of the data measurement process. To set up the problem mathematically, $\Omega$ denotes a two-dimensional homogeneous region to be inspected. Note that $\Omega$ is considered as a vacuum throughout this paper. For each $\mr\in\Omega$, the values of the background conductivity, permeability, and permittivity are set to $\sigma_0\approx0$, $\mu_0=4\pi\times \SI{e-7}{\henry/\meter}$, and $\eps_0=\SI{8.854e-12}{\farad/\meter}$, respectively, at the given angular frequency of operation $\omega=2\pi f$. Here, $k=\omega\sqrt{\eps_0\mu_0}$ denotes the lossless background wavenumber.

Assume that $\Omega$ contains a finite number of small objects $\set{D_s:s=1,2,\ldots,S}$, each of the form
\[D_s=\mr_s+\alpha_s\mB_s,\]
where $\mB_s\in\mathcal{C}^2$ is a bounded domain containing the origin. Throughout this paper, we set all objects $D_s$ as being linear, isotropic, time-invariant, and completely characterized by their permittivity value $\eps_s$ at $\omega$. We also assume that $\eps_s>\eps_0$ for all $s=1,2,\ldots,S$ and that all objects are well-separated from each other. With this setting, we introduce the following piecewise constant function of permittivity:
\[\eps(\mr)=\left\{\begin{array}{ccl}
\eps_s&\text{if}&\mr\in D_s\\
\eps_0&\text{if}&\mr\in \Omega\backslash\overline{D},
\end{array}\right.\]
where $D$ denotes the collection of objects $D_s$.

In this study, a bistatic measurement system is considered. Here, when a transmitter is fixed, the receiver rotates while measuring the data (as shown in Figure \ref{Configuration_Fresnel}). To describe the measurement configuration, $\mathcal{P}_m$ and $\mathcal{Q}_n$ denote the $m$th transmitter and $n$th receiver, respectively. In addition, $\mrp_m$ and $\mq_n$ denote the locations of $\mathcal{P}_m$ and $\mathcal{Q}_n$, respectively, such that
\[\mrp_m=P\vv_m=P(\cos\vartheta_m,\sin\vartheta_m)\qand\mq_n=Q\vt_n=Q(\cos\theta_n,\sin\theta_n),\]
where for $m=1,2,\ldots,M$ and $n=1,2,\ldots,N$
\[\vartheta_m=(m-1)\triangle\vartheta=\frac{2\pi(m-1)}{M}\qand\theta_n=(n-1)\triangle\theta=\frac{2\pi(n-1)}{N}.\]
Furthermore, $\mathcal{P}$ and $\mathcal{Q}$ denote the collection of transmitters $\mathcal{P}_m$ and receivers $\mathcal{Q}_n$, respectively.

\begin{figure}[h]
\begin{center}
\begin{tikzpicture}[scale=1.8]

\def\RT{0.65*2};
\def\RR{0.76*2};
\def\RO{0.70*2};
\def\Edge{0.15*2};

\draw[black,dashed,-] (0,0) -- ({\RT*cos(0)},{\RT*sin(0)});
\draw[gray,dashed,-] (0,0) -- ({\RT*cos(30)},{\RT*sin(30)});
\draw[gray,dashed,-] (0,0) -- ({\RR*cos(90)},{\RR*sin(90)});
\draw[gray,dashed,-] (0,0) -- ({\RR*cos(330)},{\RR*sin(330)});

\draw[black,solid,-stealth] (0.35,0) arc (0:30:0.35);
\node[black] at ({0.7*cos(70)},{0.35*sin(70)}) {$\alpha$};

\draw[black,solid,-stealth] ({0.25*cos(30)},{0.25*sin(30)}) arc (30:90:0.25);
\node[black] at ({0.55*cos(20)},{0.35*sin(20)}) {$\vartheta_m$};

\draw[green,fill=green] ({\RT*cos(30)},{\RT*sin(30)}) circle (0.05cm);
\draw[black,fill=black] ({\RT*cos(30)},{\RT*sin(30)}) circle (0.02cm);

\draw[black,solid,-stealth] ({0.15*cos(30)},{0.15*sin(30)}) arc (30:330:0.15);
\node[black] at ({-0.5*cos(0)},0) {$2\pi-\alpha$};

\foreach \beta in {90,95,...,330}
{\draw[red!20!white,fill=red!20!white] ({\RR*cos(\beta)},{\RR*sin(\beta)}) circle (0.05cm);
\draw[black!20!white,fill=black!20!white] ({\RR*cos(\beta)},{\RR*sin(\beta)}) circle (0.02cm);}



\draw[gray,solid,-stealth] ({\RO*cos(90)},{\RO*sin(90)}) arc (90:330:\RO);


\draw[red,fill=red] ({\RR*cos(200)},{\RR*sin(200)}) circle (0.05cm);
\draw[black,fill=black] ({\RR*cos(200)},{\RR*sin(200)}) circle (0.02cm);

\end{tikzpicture}\qquad
\begin{tikzpicture}[scale=1.8]

\def\RT{0.65*2};
\def\RR{0.76*2};
\def\RO{0.7*2};
\def\Edge{0.15*2};

\draw[green,fill=green] ({\RT*cos(100)},{\RT*sin(100)}) circle (0.05cm);
\draw[black,fill=black] ({\RT*cos(100)},{\RT*sin(100)}) circle (0.02cm);

\foreach \beta in {160,165,...,410}
{\draw[red,fill=red] ({\RR*cos(\beta)},{\RR*sin(\beta)}) circle (0.05cm);
\draw[black,fill=black] ({\RR*cos(\beta)},{\RR*sin(\beta)}) circle (0.02cm);}


\draw[gray,solid,-stealth] ({\RO*cos(160)},{\RO*sin(160)}) arc (160:410:\RO);

\foreach \alpha in {55,60,...,155}
{\draw[cyan!20!white,fill=cyan!20!white] ({\RR*cos(\alpha)},{\RR*sin(\alpha)}) circle (0.05cm);
\draw[black!20!white,fill=black!20!white] ({\RR*cos(\alpha)},{\RR*sin(\alpha)}) circle (0.02cm);}

\draw[green,fill=green] (-0.8,0.3) circle (0.05cm);
\draw[black,fill=black] (-0.8,0.3) circle (0.02cm);
\node[right] at (-0.8,0.3) {\text{~~transmitter}};
\draw[red,fill=red] (-0.8,0) circle (0.05cm);
\draw[black,fill=black] (-0.8,0) circle (0.02cm);
\node[right] at (-0.8,0) {\text{~~receiver for }$n\in\mathcal{I}_m$};
\draw[cyan!20!white,fill=cyan!20!white] (-0.8,-0.3) circle (0.05cm);
\draw[black!20!white,fill=black!20!white] (-0.8,-0.3) circle (0.02cm);
\node[right] at (-0.8,-0.3) {\text{~~receiver for }$n\in\mathcal{J}_m$};

\end{tikzpicture}
\caption{\label{Configuration_Fresnel}Bistatic measurement system.}
\end{center}
\end{figure}
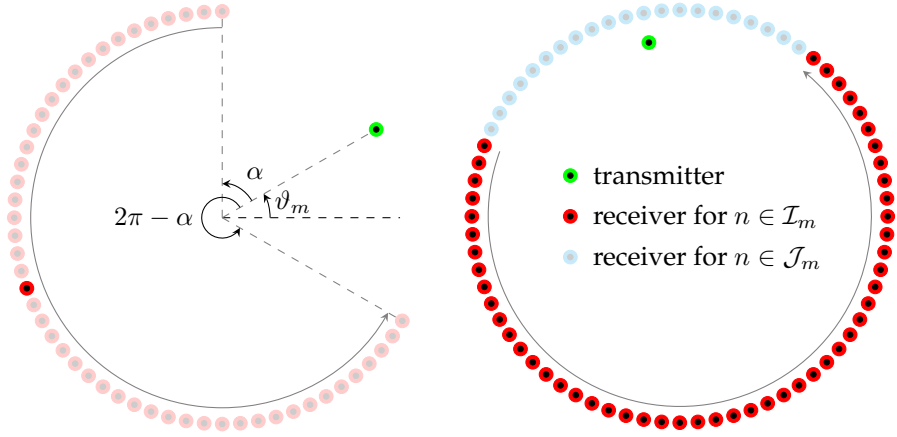

In this paper, $u(\mr,\mrp_m)$ is the solution to the Helmholtz equation in the presence of $D$
\[\triangle u(\mr,\mrp_m)+\omega^2\eps(\mr)\mub u(\mr,\mrp_m)=0,\quad\mr\in\Omega\]
with transmission condition at $\p D_s$. The incident field $u_0(\mr,\mrp_m)$ generated at $\mathcal{P}_m$ satisfies the following equation:
\[\triangle u_0(\mr,\mrp_m)+k^2 u_0(\mr,\mrp_m)=0,\quad\mr\in\Omega.\]
Here, the time harmonic dependence $e^{-i\omega t}$ is assumed, $k^2$ is not an eigenvalue for the operator $-\triangle$, and\[u_{0}(\mr,\mr')=G(\mr,\mr')=-\frac{i}{4}H_0^{(1)}(k|\mr-\mr'|),\quad\mr\ne\mr',\]
where $H_0^{(1)}$ denotes the zero order Hankel function of the first kind. $u_{\scat}(\mq_n,\mrp_m)$ denotes the scattered field measured at the $\mathcal{Q}_n$ corresponding to the incident field that satisfies $u_{\scat}(\mq_n,\mrp_m)=u(\mq_n,\mrp_m)-u_0(\mq_n,\mrp_m)$ and the Sommerfeld radiation condition:
\[\lim_{|\mr|\to\infty}\sqrt{|\mr|}\left(\frac{\p u_{\scat}(\mr,\mrp_m)}{\p|\mr|}-i ku_{\scat}(\mr,\mrp_m)\right)=0\quad\text{uniformly in all directions}\quad\frac{\mr}{|\mr|}.\]
Let us emphasize that owing to \cite{CK}, $u_{\scat}(\mq_n,\mrp_m)$ can be represented as the following integral equation formula with an unknown density function $\varphi$:
\begin{equation}\label{ScatteredField-SLP}
u_{\scat}(\mq_n,\mrp_m)=\int_DG(\mq_n,\mr')\varphi(\mr',\mrp_m)\rd\mr'.
\end{equation}
Note that a complete expression of $u_{\scat}(\mq_n,\mrp_m)$ is required to design an indicator function for the DSM; however, the complete form of $\varphi$ is unknown without a priori information of $D$. Thus, based on the findings of a previous study \cite{BCS}, we adopt the following alternative representation formula:
\begin{equation}\label{ScatteredField}
u_{\scat}(\mq_n,\mrp_m)\approx k^2\int_D\mathcal{O}(\mr')G(\mq_n,\mr')G(\mr',\mrp_m)\rd\mr',\quad\mathcal{O}(\mr')=\frac{\eps(\mr')-\epsb}{\epsb\mub}.
\end{equation}

\section{Indicator Function of the DSM with a Single Source}\label{sec:3}
\subsection{Introduction to the Indicator Function}
Here, we introduce the indicator function of the DSM with a single source. For a fixed transmitter $\mathcal{P}_m$, we generate an arrangement of measurement data such that
\begin{equation}\label{Measurement_Single}
\mathbb{S}(m)=\begin{bmatrix}
u_{\scat}(\mq_1,\mrp_m)\\
u_{\scat}(\mq_2,\mrp_m)\\
\vdots\\
u_{\scat}(\mq_N,\mrp_m)
\end{bmatrix}.
\end{equation}
By applying the mean-value theorem to \eqref{ScatteredField}, $u_{\scat}(\mq_n,\mrp_m)$ can be written as follows:
\begin{equation}\label{ScatteredField-MVT}
u_{\scat}(\mq_n,\mrp_m)\approx k^2\sum_{s=1}^{S}\left(\frac{\eps_s-\epsb}{\epsb\mub}\right)\area(D_s)G(\mrp_m,\mr_s)G(\mq_n,\mr_s).
\end{equation}
Correspondingly, $\mathbb{S}(m)$ becomes:
\begin{align*}
\mathbb{S}(m)&\approx k^2\sum_{s=1}^{S}\left(\frac{\eps_s-\epsb}{\epsb\mub}\right)\area(D_s)G(\mrp_m,\mr_s)
\begin{bmatrix}
G(\mq_1,\mr_s)\\
G(\mq_2,\mr_s)\\
\vdots\\
G(\mq_N,\mr_s)
\end{bmatrix}\\
&:=k^2\sum_{s=1}^{S}\left(\frac{\eps_s-\epsb}{\epsb\mub}\right)\area(D_s)G(\mrp_m,\mr_s)\mathbb{G}(\mr_s).
\end{align*}

Based on the above expression, we generate a test vector corresponding to the receivers. Here, for each search point $\mr\in\Omega$, we generate:
\begin{equation}\label{TestVector_Receiver}
\mathbb{Q}(\mr)=\begin{bmatrix}
G(\mq_1,\mr)\\
G(\mq_2,\mr)\\
\vdots\\
G(\mq_N,\mr)
\end{bmatrix}.
\end{equation}
Then, based on the orthogonality property of the Hilbert space $\ell^2$, the value
\begin{equation}\label{Orthogonality_Property}
\langle\mG(\mr_s),\mQ(\mr)\rangle_{\ell^2(\mathcal{Q})}:=\mG(\mr_s)\cdot\overline{\mQ(\mr)}=\sum_{n=1}^{N}G(\mq_n,\mr_s)\overline{G(\mq_n,\mr)}
\end{equation}
reaches its maximum when $\mr=\mr_s\in D_s$. Correspondingly, by defining the norm $\|\cdot\|_{\ell^2(\mathcal{Q})}=\langle\cdot,\cdot\rangle_{\ell^2(\mathcal{Q})}^{1/2}$, the following classical indicator function of the DSM can be introduced:
\begin{equation}\label{Classical_Indicator}
\mathfrak{F}(\mr,m)=\frac{|\langle\mS(m),\mQ(\mr)\rangle_{\ell^2(\mathcal{Q})}|}{\|\mS(m)\|_{\ell^2(\mathcal{Q})}\|\mQ(\mr)\|_{\ell^2(\mathcal{Q})}}.
\end{equation}
Then, $\mathfrak{F}(\mr,m)\approx1$ when $\mr=\mr_s\in D_s$, $s=1,2,\ldots,S$; thus, by regarding peaks of large magnitudes in the map of $\mathfrak{F}(\mr,m)$, it is possible to identify the objects $D_s$. Refer to the literature \cite{IJZ1,KLP1} for a more detailed description.

However, $u_{\scat}(\mq_n,\mrp_m)$ cannot be expressed as \eqref{ScatteredField} due to the interference between the antennas because using the classical indicator function in \eqref{Classical_Indicator} in this case is nonsense. To avoid such interference, when a transmitter is placed at a fixed position $\mrp_m$, a receiver is rotating within the range $\vartheta_m+\alpha$ to $\vartheta_m+2\pi-\alpha$ for $0<\alpha<\pi$ (in \cite{BS}, the value of $\alpha$ is set to $\SI{60}{\degree}$), as shown in Figure \ref{Configuration_Fresnel}. Here, $\alpha$ denotes the minimum bistatic angle, i.e., the angle between the transmitter and the receiver. This means that the complete elements of $\mS(m)$ of \eqref{Measurement_Single} cannot be used to design the indicator function in \eqref{Classical_Indicator}. Now, we define the following two index sets:
\begin{align*}
\mathcal{I}(m)&=\set{n:\vartheta_m+\alpha\leq\theta_n\leq\vartheta_m+2\pi-\alpha},\\
\mathcal{J}(m)&=\set{n:\vartheta_m-\alpha<\theta_n<\vartheta_m+\alpha}=\set{1,2,\ldots,N}\backslash\mathcal{I}(m).
\end{align*}
Then, it is possible to collect $u_{\scat}(\mq_n,\mrp_m)$ for $n\in\mathcal{I}(m)$. Thus, by converting uncollectible data to a fixed constant $C$, we generate an arrangement of measurement data such that
\[\mathbb{S}(C,m)=[a_{nm}],\quad\text{where}\quad a_{nm}=\left\{\begin{array}{cl}
\smallskip u_{\scat}(\mq_n,\mrp_m),&n\in\mathcal{I}(m)\\
C,&n\in\mathcal{J}(m).
\end{array}\right.\]
Figure \ref{MSR} illustrates matrix $\mathbb{K}$, where
\[\mathbb{K}=\Big[|\mS(0,1)|\quad|\mS(0,2)|\quad\cdots\quad|\mS(0,M)|\Big].\]
With this, we introduce the following indicator function of the DSM. Here, for each $\mr\in\Omega$ and fixed constant $C$, we obtain:
\begin{equation}\label{IndicatorFunction_Single}
\mathfrak{F}_{\dsm}(\mr,m,C)=\frac{|\langle\mS(C,m),\mQ(\mr)\rangle_{\ell^2(\mathcal{Q})}|}{\|\mS(C,m)\|_{\ell^2(\mathcal{Q})}\|\mQ(\mr)\|_{\ell^2(\mathcal{Q})}}.
\end{equation}
Note that the test vector $\mQ(\mr)$ is given in \eqref{TestVector_Receiver}.

\begin{figure}[h]
\includegraphics[width=.33\textwidth]{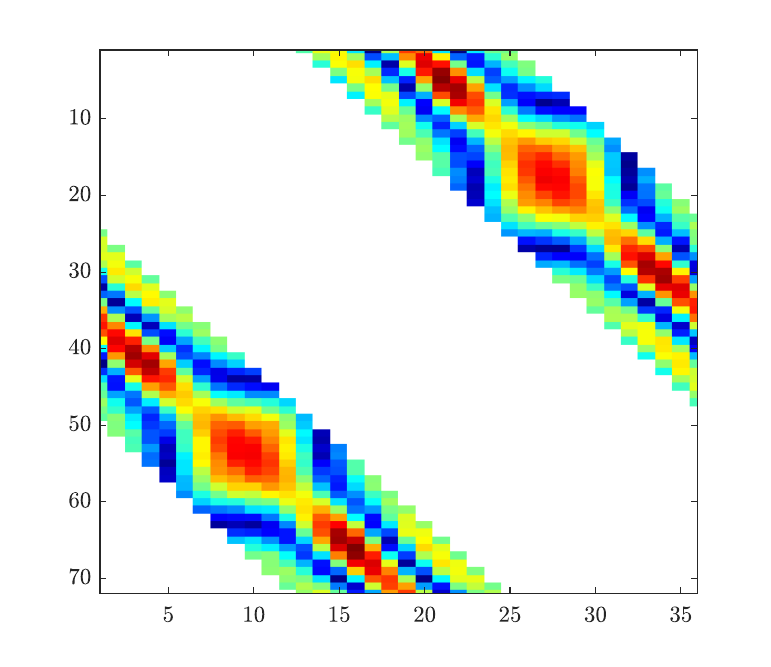}\hfill
\includegraphics[width=.33\textwidth]{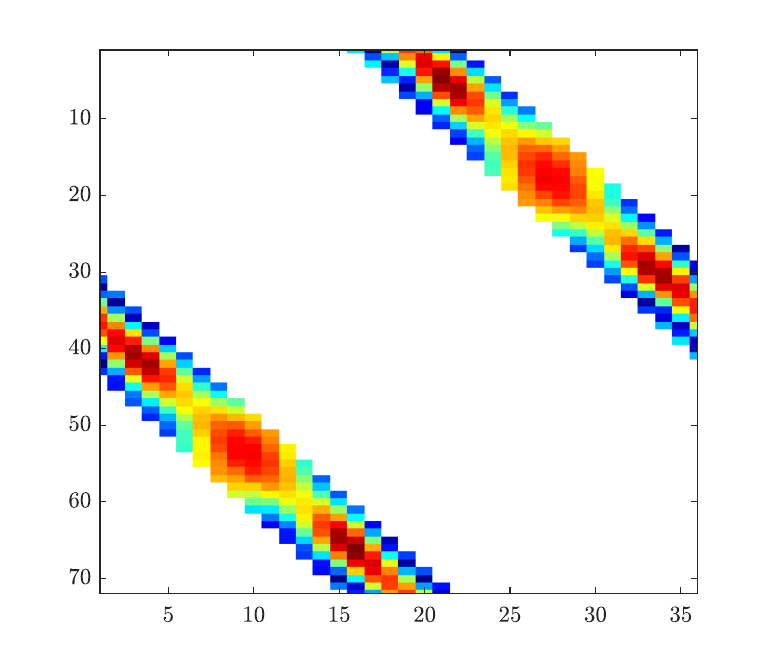}\hfill
\includegraphics[width=.33\textwidth]{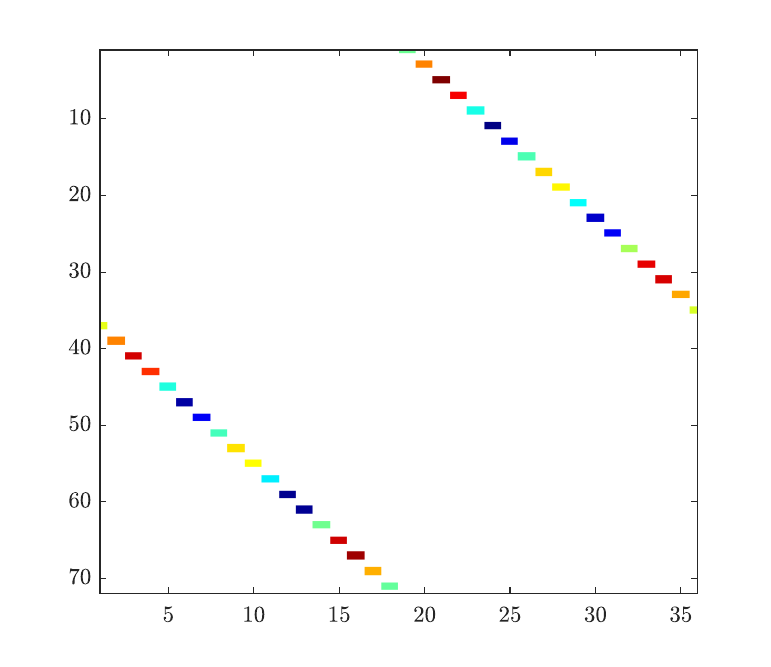}
\caption{\label{MSR}Visualization of the modulus of $\mathbb{K}$ for $M=36$ and $N=72$ with $\alpha=\SI{120}{\degree}$ (left), $\alpha=\SI{150}{\degree}$ (middle), and $\alpha=\SI{180}{\degree}$ (right).}
\end{figure}

\subsection{Mathematical Structure of the Indicator Function}
Throughout the numerical simulation results discussed in Section \ref{sec:5}, the imaging performance is strongly dependent on the location of $\mathcal{P}_m$ and the value of $C$. To explain this phenomenon theoretically, we investigate the mathematical structure of the indicator function as follows.

\begin{Theorem}\label{Theorem_Single}
Here, assume that $4k|\mr-\mq_n|\gg1$ for all $\mr\in\Omega$ and $n=1,2,\ldots,N$. Then, $\mathfrak{F}_{\dsm}(\mr,m,C)$ can be expressed as follows:
\begin{equation}\label{Structure_Single}
\mathfrak{F}_{\dsm}(\mr,m,C)\approx\frac{|\Phi(\mr,\mr')+\Psi(\mr)|}{\displaystyle\max_{\mr\in\Omega}|\Phi(\mr,\mr')+\Psi(\mr)|},
\end{equation}
where
\[\Phi(\mr,\mr')=\frac{k^2\#[\mathcal{I}(m)]}{2\sqrt{kQ\pi}}\int_D\mathcal{O}(\mr')G(\mrp_m,\mr')\Big(J_0(k|\mr-\mr'|)+\mathcal{E}_1(\mr,\mr',m)\Big)\rd\mr'\]
and
\[\Psi(\mr)=-C(1+i)\#[\mathcal{J}(m)]\Big(J_0(k|\mr|)+\mathcal{E}_2(\mr,m)\Big).\]
Here, $\#[\mathcal{I}(m)]$ denotes the number of elements in the set $\mathcal{I}(m)$, and $J_n$ is the Bessel function of order $n$. In addition, we obtain the following:
\begin{align}
\begin{aligned}\label{DisturbFactor}
&\mathcal{E}_1(\mr,\mr',m)=2\sum_{p=1}^{\infty}i^pJ_p(k|\mr-\mr'|)\cos\big(p(\vartheta_m-\phi)\big)\sinc\big(p(\pi-\alpha)\big)\\
&\mathcal{E}_2(\mr,m)=2\sum_{q=1}^{\infty}i^qJ_q(k|\mr|)\cos\big(q(\vartheta_m-\psi)\big)\sinc(q\alpha),
\end{aligned}
\end{align}
where $\sinc(x)$ denotes the unnormalized sinc function defined for $x\ne0$ by
\[\sinc(x)=\frac{\sin x}{x}.\]
\end{Theorem}
\begin{proof}
Since $4k|\mr-\mq_n|\gg1$ for all $n$, the following asymptotic form holds (refer to the \cite{CK} for additional information):
\begin{equation}\label{Asymptotic_Hankel}
-\frac{i}{4}H_0^{(1)}(k|\mr-\mq_n|)=-\frac{(1+i)e^{ik Q}}{4\sqrt{kQ\pi}}e^{-ik\vt_n\cdot\mr}.
\end{equation}
Then, since
\[\mQ(\mr)=-\frac{(1+i)e^{-ikQ}}{4\sqrt{kQ\pi}}
\begin{bmatrix}
e^{-ik\vt_1\cdot\mr}\\
e^{-ik\vt_2\cdot\mr}\\
\vdots\\
e^{-ik\vt_N\cdot\mr}
\end{bmatrix}\]
and
\[u_{\scat}(\mq_n,\mrp_m)=-\frac{k^2(1+i)e^{ik Q}}{4\sqrt{kQ\pi}}\int_D\mathcal{O}(\mr')G(\mr',\mrp_m)e^{-ik\vt_n\cdot\mr'}\rd\mr',\]
we can evaluate:
\begin{align*}
\langle\mS(C,m),\mQ(\mr)\rangle_{\ell^2(\mathcal{Q})}=&\frac{ik^2}{8kQ\pi}\int_D\mathcal{O}(\mr')G(\mr',\mrp_m)\sum_{n\in\mathcal{I}(m)}e^{ik\vt_n\cdot(\mr-\mr')}\rd\mr'\\
&-\frac{C(1+i)e^{-ikQ}}{4\sqrt{kQ\pi}}\sum_{n\in\mathcal{J}(m)}e^{ik\vt_n\cdot\mr}
\end{align*}

The following relation holds uniformly (refer to the literature \cite{P-SUB3} for the derivation).
\begin{align}
\begin{aligned}\label{Jacobi--Anger}
&\sum_{n=1}^{N}e^{ix\cos(\theta_n-\phi)}=\frac{N}{N\triangle\theta}\sum_{n=1}^{N}e^{ix\cos(\theta_n-\phi)}\triangle\theta\approx\frac{N}{\theta_N-\theta_1}\int_{\theta_1}^{\theta_N}e^{ix\cos(\theta-\phi)}\rd\theta\\
&=N\left(J_0(x)+\frac{4}{\theta_N-\theta_1}\sum_{p=1}^{\infty}\frac{i^p}{p}J_p(x)\cos\frac{p(\theta_N+\theta_1-2\phi)}{2}\sin\frac{p(\theta_N-\theta_1)}{2}\right)\\
&=N\left(J_0(x)+2\sum_{p=1}^{\infty}i^pJ_p(x)\cos\frac{p(\theta_N+\theta_1-2\phi)}{2}\sinc\big(p(\theta_N-\theta_1)\big)\right).
\end{aligned}
\end{align}
Thus, by letting $\mr-\mr'=|\mr-\mr'|(\cos\phi,\sin\phi)$ and $\mr=|\mr|(\cos\psi,\sin\psi)$, we can evaluate
\begin{align*}
\sum_{n\in \mathcal{I}(m)}&e^{ik\vt_n\cdot(\mr-\mr')}=\frac{\#[\mathcal{I}(m)]}{\#[\mathcal{I}(m)]\triangle\theta}\sum_{n\in \mathcal{I}(m)}e^{ik\vt_n\cdot(\mr-\mr')}\triangle\theta\\
&\approx\frac{\#[\mathcal{I}(m)]}{2\pi-2\alpha}\int_{\vartheta_m+\alpha}^{\vartheta_m+2\pi-\alpha}e^{ik|\mr-\mr'|\cos(\theta-\phi)}\rd\theta\\
&=\#[\mathcal{I}(m)]\left[J_0(k|\mr-\mr'|)+\frac{2}{\pi-
\alpha}\sum_{p=1}^{\infty}\frac{i^p}{p}J_p(k|\mr-\mr'|)\cos\big(p(\vartheta_m-\phi)\big)\sin\big(p(\pi-\alpha)\big)\right]\\
&=\#[\mathcal{I}(m)]\left[J_0(k|\mr-\mr'|)+2\sum_{p=1}^{\infty}i^pJ_p(k|\mr-\mr'|)\cos\big(p(\vartheta_m-\phi)\big)\frac{\sin\big(p(\pi-\alpha)\big)}{p(\pi-\alpha)}\right]
\end{align*}
and
\begin{align*}
\sum_{n\in\mathcal{J}(m)}e^{ik\vt_n\cdot\mr}&=\frac{\#[\mathcal{J}(m)]}{\#[\mathcal{J}(m)]\triangle\theta}\sum_{n\in\mathcal{J}(m)}e^{ik\vt_n\cdot\mr}\triangle\theta\approx\frac{\#[\mathcal{J}(m)]}{2\alpha}\int_{\vartheta_m-\alpha}^{\vartheta_m+\alpha}e^{ik|\mr|\cos(\theta-\psi)}\rd\theta\\
&=\#[\mathcal{J}(m)]\left[J_0(k|\mr|)+\frac{2}{\alpha}\sum_{q=1}^{\infty}\frac{i^q}{q}J_q(k|\mr|)\cos\big(q(\vartheta_m-\psi)\big)\sin(q\alpha)\right]\\
&=\#[\mathcal{J}(m)]\left[J_0(k|\mr|)+2\sum_{q=1}^{\infty}i^qJ_q(k|\mr|)\cos\big(q(\vartheta_m-\psi)\big)\frac{\sin(q\alpha)}{q\alpha}\right].
\end{align*}
Therefore, we can derive
\begin{multline}\label{Structure_InnerProduct}
\langle\mS(C,m),\mQ(\mr)\rangle_{\ell^2(\mathcal{Q})}=\frac{ik^2\#[\mathcal{I}(m)]}{8kQ\pi}\int_D\mathcal{O}(\mr')G(\mr',\mrp_m)\Big(J_0(k|\mr-\mr'|)+\mathcal{E}_1(\mr,\mr',m)\Big)\rd\mr'\\
-\frac{C(1+i)\#[\mathcal{J}(m)]e^{-ikQ}}{4\sqrt{kQ\pi}}\Big(J_0(k|\mr|)+\mathcal{E}_2(\mr,m)\Big).
\end{multline}
Finally, the structure given in \eqref{Structure_Single} can be obtained by applying H{\"o}lder's inequality as follows
\[\langle\mS(C,m),\mQ(\mr)\rangle_{\ell^2(\mathcal{Q})}\leq\|\mS(C,m)\|_{\ell^2(\mathcal{Q})}\|\mQ(\mr)\|_{\ell^2(\mathcal{Q})}.\]
\end{proof}

\subsection{Properties of Indicator Function with a Single Source}
Following the result in Theorem \ref{Theorem_Single}, here, we discuss some properties of the indicator function with a single source.

\begin{Discussion}[Composition of Indicator Function]\label{discussion1}
Since $J_0(0)=1$ and $J_p(0)=0$ for $p\ne0$, the factor $J_0(k|\mr-\mr'|)$ contributes to the object detection because $J_0(k|\mr-\mr'|)=1$ when $\mr=\mr'\in D$. However, due to the factor $G(\mrp_m,\mr')$, the imaging performance is strongly dependent on the location of the transmitter. In contrast, the factor $\mathcal{E}_1(\mr,\mr',m)$ disturbs the object detection because $\mathcal{E}_1(\mr,\mr',m)=0$ when $\mr=\mr'\in D$ and generates several artifacts based on the oscillation property of the Bessel function. In addition, due to the factor $J_0(k|\mr|)$, a peak of large magnitude will appear at the origin, which disturbs the object detection (unless the object is located at the origin). Finally, the factor $\mathcal{E}_2(\mr,m)$ generates several artifacts and does not contribute to the object detection.
\end{Discussion}

\begin{Discussion}[Influence of Bistatic Angle]\label{discussion2}
Based on the identified structure \eqref{Structure_Single}, we can observe that the imaging performance of $\mathfrak{F}_{\dsm}(\mr,m,C)$ is strongly dependent on $\alpha$. If $\alpha$ is sufficiently small such that $\alpha\longrightarrow0+$, then since $\#[\mathcal{J}(m)]\longrightarrow0+$ and $\sinc(p\pi)=0$ for any integer $p$,
\[\lim_{\alpha\to0+}\mathcal{E}_1(\mr,\mr',m)=2\sum_{p=1}^{\infty}i^pJ_p(k|\mr-\mr'|)\cos\big(p(\vartheta_m-\phi)\big)\bigg(\lim_{\alpha\to0+}\sinc\big(p(\pi-\alpha)\big)\bigg)=0,\]
the factor $\Phi(\mr,\mr')$ becomes
\[\Phi(\mr,\mr')=\frac{k^2\#[\mathcal{I}(m)]}{2\sqrt{kQ\pi}}\int_D\mathcal{O}(\mr')G(\mrp_m,\mr')J_0(k|\mr-\mr'|)\rd\mr'.\]
Thus, a good imaging result is expected when $\alpha$ is small because the disturbing factors are eliminated. Otherwise, if $\alpha\longrightarrow\pi-$, then since
\[\lim_{\alpha\to\pi-}\sinc\big(p(\pi-\alpha)\big)=1,\]
based on the uniform convergence of the Jacobi--Anger expansion formula (refer to the literature \cite{CK} for additional details)
\begin{equation}\label{Jacobi--Anger}
e^{ix\cos\vartheta}=J_0(x)+2\sum_{p=1}^{\infty}i^pJ_p(x)\cos(p\vartheta),
\end{equation}
we obtain
\[J_0(k|\mr-\mr'|)+2\sum_{p=1}^{\infty}i^pJ_p(k|\mr-\mr'|)\cos\big(p(\vartheta_m-\phi)\big)=e^{ik|\mr-\mr'|\cos(\vartheta_m-\phi)}=e^{ik\vv_m\cdot(\mr-\mr')}\]
and correspondingly, the factor $\Phi(\mr,\mr')$ of \eqref{Structure_Single} becomes
\begin{align*}
\Phi(\mr,\mr')&=\frac{k^2\#[\mathcal{I}(m)]}{2\sqrt{kQ\pi}}\int_D\mathcal{O}(\mr')G(\mrp_m,\mr')e^{ik\vv_m\cdot(\mr-\mr')}\rd\mr'.
\end{align*}
In addition, since $|e^{ik\vv_m\cdot(\mr-\mr')}|\equiv1$ for every $\mr\in\Omega$, the factor $\Phi(\mr,\mr')$ no longer contributes to identifying the objects $D$. Consequently, it is impossible to identify the objects through the map $\mathfrak{F}_{\dsm}(\mr,m,C)$ when $\alpha$ is close to $\pi$.
\end{Discussion}

\begin{Discussion}[Influence of Converted Constant]\label{discussion3}
Here, we assume that $0<\alpha<\pi$ and $\mr=\mr_s\in D_s$ for some $s$. Then, based on Discussion \ref{discussion1},
\[\Phi(\mr,\mr')=\frac{k^2\#[\mathcal{I}(m)]}{2\sqrt{kQ\pi}}\int_{D_s}\left(\frac{\eps_s-\epsb}{\epsb\mub}\right)G(\mrp_m,\mr')\rd\mr'-C(1+i)\#[\mathcal{J}(m)]\Big(J_0(k|\mr|)+\mathcal{E}_2(\mr,m)\Big).\]
Thus, if $C$ satisfies
\[|C|\ll\frac{k^2\#[\mathcal{I}(m)]}{2\#[\mathcal{J}(m)]\sqrt{2kQ\pi}}\left(\frac{\eps_s-\epsb}{\epsb\mub}\right)\frac{G(\mrp_m,\mr_s)\area(D_s)}{J_0(k|\mr_s|)+\mathcal{E}_2(\mr_s,m)},\]
then the factor $|\Psi(\mr)|$ is dominated by $|\Phi(\mr,\mr')|$. As a result, it is possible to recognize $D_s$ through the map of $\mathfrak{F}_{\dsm}(\mr,m,C)$. Otherwise, if $C$ satisfies the following relation for all $s=1,2,\ldots,S$:
\[|C|\gg\frac{k^2\#[\mathcal{I}(m)]}{2\#[\mathcal{J}(m)]\sqrt{2kQ\pi}}\left(\frac{\eps_s-\epsb}{\epsb\mub}\right)\frac{G(\mrp_m,\mr_s)\area(D_s)}{J_0(k|\mr_s|)+\mathcal{E}_2(\mr_s,m)},\]
then the factor $|\Phi(\mr,\mr')|$ is dominated by $|\Psi(\mr)|$. This means that the map of $\mathfrak{F}_{\dsm}(\mr,m,C)$ will only contain a peak of large magnitude at the origin. Thus, it will be very difficult to recognize the existence, location, and shape of the objects.

Based on the above observations, the selection of $C$ to identify $D_s$ is highly dependent on the $\eps_s$ values, the size of $D_s$, $k$, and the total number of transmitting and receiving antennas.
\end{Discussion}

\begin{Discussion}[Best Choice for the Constant]\label{discussion4}
Based on both Discussions \ref{discussion2} and \ref{discussion3}, $C=0$ and $\alpha\longrightarrow0$ will guarantee good imaging results. However, the setting $\alpha\longrightarrow0$ is impossible; thus, $C=0$ is the best choice for a proper application of the designed DSM, which is the theoretical rationale for converting unmeasurable data to zero in previous studies.
\end{Discussion}

\section{Indicator Function of DSM with Multiple Sources}\label{sec:4}
Further improvement of the indicator with single source is required because the imaging performance is strongly dependent on the location of the transmitter (refer to Discussion \ref{discussion1}) and the selection of the constant $C$, which is in turn highly dependent on the given problem at hand (e.g., the material properties of unknown objects; refer to Discussion \ref{discussion3}). Thus, in the following, we introduce an indicator function with multiple sources.

\subsection{Introduction to the Indicator Function}
For each transmitter $\mathcal{P}_m$, $m=1,2,\ldots,M$, we generate an arrangement such that
\[\mathbb{M}(C)=\begin{bmatrix}
\langle\mM(C,1),\mQ(\mr)\rangle_{\ell^2(\mathcal{Q})}\\
\langle\mM(C,2),\mQ(\mr)\rangle_{\ell^2(\mathcal{Q})}\\
\vdots\\
\langle\mS(C,M),\mQ(\mr)\rangle_{\ell^2(\mathcal{Q})}
\end{bmatrix}.\]
Here, if $C=0$, since $\#[\mathcal{I}(m)]$ are identical, then based on the structure \eqref{Structure_InnerProduct},
\[\langle\mM(0,m),\mQ(\mr)\rangle_{\ell^2(\mathcal{Q})}=\frac{ik^2\#[\mathcal{I}(m)]}{8kQ\pi}\int_D\mathcal{O}(\mr')G(\mrp_m,\mr')\Big(J_0(k|\mr-\mr'|)+\mathcal{E}_1(\mr,\mr',m)\Big)\rd\mr',\]
the arrangement $\mM(C)$ can be written as
\[\mathbb{M}(C)=\frac{ik^2\#[\mathcal{I}(m)]}{8kQ\pi}\begin{bmatrix}
\smallskip\displaystyle\int_D\mathcal{O}(\mr')G(\mrp_1,\mr')\Big(J_0(k|\mr-\mr'|)+\mathcal{E}_1(\mr,\mr',m)\Big)\rd\mr'\\
\displaystyle\int_D\mathcal{O}(\mr')G(\mrp_2,\mr')\Big(J_0(k|\mr-\mr'|)+\mathcal{E}_1(\mr,\mr',m)\Big)\rd\mr'\\
\vdots\\
\displaystyle\int_D\mathcal{O}(\mr')G(\mrp_M,\mr')\Big(J_0(k|\mr-\mr'|)+\mathcal{E}_1(\mr,\mr',m)\Big)\rd\mr'
\end{bmatrix}.\]
Based on the above, the factor $G(\mrp_m,\mr')J_0(k|\mr-\mr'|)$ contains information of $\mr'\in D$; therefore, similar to the development of the indicator function with a single source, we generate a test vector corresponding to the transmitters. Here, for each search point $\mr\in\Omega$,
\begin{equation}\label{TestVector_Transmitter}
\mathbb{P}(\mr)=\begin{bmatrix}
G(\mrp_1,\mr)\\
G(\mrp_2,\mr)\\
\vdots\\
G(\mrp_M,\mr)
\end{bmatrix}.
\end{equation}
Then, to test the orthogonality property of the Hilbert space $\ell^2$, we define
\[\langle\mM(C),\mP(\mr)\rangle_{\ell^2(\mathcal{P})}:=\mM(C)\cdot\overline{\mP(\mr)}\qand\|\cdot\|_{\ell^2(\mathcal{P})}=\langle\cdot,\cdot\rangle_{\ell^2(\mathcal{P})}^{1/2},\]
and the following indicator function of the DSM with multiple sources can be introduced. Here, for each $\mr\in\Omega$ and fixed constant $C$, we obtain
\begin{equation}\label{IndicatorFunction_Multiple}
\mathfrak{F}_{\msm}(\mr,C)=\frac{|\langle\mM(C),\mP(\mr)\rangle_{\ell^2(\mathcal{P})}|}{\|\mM(C)\|_{\ell^2(\mathcal{P})}\|\mP(\mr)\|_{\ell^2(\mathcal{P})}}.
\end{equation}

\subsection{Mathematical Structure of the Indicator Function}
Based on numerical simulation results discussed in Section \ref{sec:5}, we can say that $\mathfrak{F}_{\msm}(\mr,C)$ with $C=0$ improves the imaging performance. To support this fact theoretically, we investigate the mathematical structure of the $\mathfrak{F}_{\msm}(\mr,C)$ as follows.

\begin{Theorem}\label{Theorem_Multiple}
Here, we assume that $4k|\mr-\mr_m|\gg1$ and $4k|\mr-\mq_n|\gg1$ for all $\mr\in\Omega$, $m=1,2,\ldots,M$, and $n=1,2,\ldots,N$. Then, $\mathfrak{F}_{\msm}(\mr,C)$ can be expressed as follows:
\begin{equation}\label{Structure_Multiple}
\mathfrak{F}_{\msm}(\mr,C)\approx\frac{|\Lambda(\mr,\mr')+\Gamma(\mr)|}{\displaystyle\max_{\mr\in\Omega}|\Lambda(\mr,\mr')+\Gamma(\mr)|},
\end{equation}
where
\[\Lambda(\mr,\mr')=\frac{\#[\mathcal{I}(m)]}{8\sqrt{PQ}\pi}\int_D\mathcal{O}(\mr')\bigg(J_0(k|\mr-\mr'|)^2+2\sum_{p=1}^{\infty}(-1)^pJ_p(k|\mr-\mr'|)^2\sinc\big(p(\pi-\alpha)\big)\bigg)\rd\mr'\]
and
\[\Gamma(\mr)=\frac{C\#[\mathcal{J}(m)]}{k}\left(J_0(k|\mr|)^2+2\sum_{q=1}^{\infty}(-1)^qJ_q(k|\mr|)^2\sinc(q\alpha)\right).\]
\end{Theorem}
\begin{proof}
Based on \eqref{Asymptotic_Hankel}, \eqref{Structure_InnerProduct}, and \eqref{TestVector_Transmitter}, we obtain
\begin{align}
\begin{aligned}\label{Term}
\langle\mM&(C),\mP(\mr)\rangle_{\ell^2(\mathcal{P})}\\
=&-\sum_{m=1}^{M}\frac{\#[\mathcal{I}(m)]}{64PQ\pi^2}\int_D\mathcal{O}(\mr')e^{ik\vv_m\cdot(\mr-\mr')}\Big(J_0(k|\mr-\mr'|)+\mathcal{E}_1(\mr,\mr',m)\Big)\rd\mr'\\
&-\sum_{m=1}^{M}\frac{iC\#[\mathcal{J}(m)]e^{-ik(P+Q)}}{8k\pi\sqrt{PQ}}e^{ik\vv_m\cdot\mr}\Big(J_0(k|\mr|)+\mathcal{E}_2(\mr,m)\Big).
\end{aligned}
\end{align}
First, by applying \eqref{Jacobi--Anger}, we immediately obtain
\begin{equation}\label{Term1}
\sum_{m=1}^{M}e^{ik\vv_m\cdot(\mr-\mr')}J_0(k|\mr-\mr'|)\approx MJ_0(k|\mr-\mr'|)^2.
\end{equation}
Next, by letting $\mr-\mr'=|\mr-\mr'|(\cos\phi,\sin\phi)$ and $\mr=|\mr|(\cos\psi,\sin\psi)$, based on the \eqref{Jacobi--Anger}, we derive
\begin{align}
\begin{aligned}\label{Term2}
&\sum_{m=1}^{M}e^{ik\vv_m\cdot(\mr-\mr')}\cos\big(q(\vartheta_m-\phi)\big)\approx\frac{M}{2\pi}\int_0^{2\pi}e^{ik|\mr-\mr'|\cos(\varphi-\phi)}\cos\big(q(\vartheta-\psi)\big)\rd\vartheta\\
&=\frac{M}{2\pi}\int_0^{2\pi}\left(J_0(k|\mr-\mr'|)+2\sum_{p=1}^{\infty}i^pJ_p(k|\mr-\mr'|)\cos\big(p(\vartheta-\phi)\big)\right)\cos\big(q(\vartheta-\phi)\big)\rd\vartheta\\
&=Mi^qJ_q(k|\mr-\mr'|)
\end{aligned}
\end{align}
and
\begin{equation}\label{Term3}
\sum_{m=1}^{M}e^{ik\vv_m\cdot\mr}J_0(k|\mr|)\approx MJ_0(k|\mr|)^2.
\end{equation}
Similar to the derivation in \eqref{Term2}, we can examine
\begin{equation}\label{Term4}
\sum_{m=1}^{M}e^{ik\vv_m\cdot\mr}\cos\big(q(\vartheta_m-\psi)\big)\approx Mi^qJ_q(k|\mr|).
\end{equation}
Thus, plugging \eqref{Term1}--\eqref{Term4} into \eqref{Term} yields
\begin{align*}
\langle\mM&(C),\mP(\mr)\rangle_{\ell^2(\mathcal{P})}\\
=&-\frac{M\#[\mathcal{I}(m)]}{64PQ\pi^2}\int_D\mathcal{O}(\mr')\left(J_0(k|\mr-\mr'|)^2+2\sum_{p=1}^{\infty}(-1)^pJ_p(k|\mr-\mr'|)^2\sinc\big(p(\pi-\alpha)\big)\right)\rd\mr'\\
&-\frac{iCM\#[\mathcal{J}(m)]e^{-ik(P+Q)}}{8k\pi\sqrt{PQ}}\left(J_0(k|\mr|)^2+2\sum_{q=1}^{\infty}(-1)^qJ_q(k|\mr|)^2\sinc(q\alpha)\right).
\end{align*}
Finally, by applying H{\"o}lder's inequality, we obtain \eqref{Structure_Multiple}.
\end{proof}

\subsection{Properties of Indicator Function With Multiple Sources}
Following the result in Theorem \ref{Theorem_Multiple}, in the following, we discuss some properties of the indicator function with multiple sources.

\begin{Discussion}[Composition of Indicator Function]\label{discussion5}
Similar to the indicator function with a single source, the factor $J_0(k|\mr-\mr'|)^2$ contributes to the object detection while the remaining factors disturb the object detection and generate several artifacts. In contrast, the imaging performance of the $\mathfrak{F}_{\msm}(\mr,C)$ is independent of the transmitter’s location. To compare the imaging performance between single and multiple sources, we consider the following:
\begin{align*}
f_1(x)&=J_0(k|x|)+2\sum_{p=1}^{10^5}i^pJ_p(k|x|)\sinc\left(\frac{p\pi}{2}\right)\\
f_2(x)&=J_0(k|x|)^2+2\sum_{p=1}^{10^5}(-1)^pJ_p(k|x|)^2\sinc\left(\frac{p\pi}{2}\right).
\end{align*}
Note that these are similar to the 1D version of $\mathfrak{F}_{\dsm}(\mr,m,0)$ and $\mathfrak{F}_{\msm}(\mr,0)$ with $\alpha=\SI{90}{\degree}$ in the presence of a single object located at the origin. By comparing the plots, we can say that the application of multiple sources will yield better images because less oscillation is involved than when imaging is performed with only a single source.
\end{Discussion}

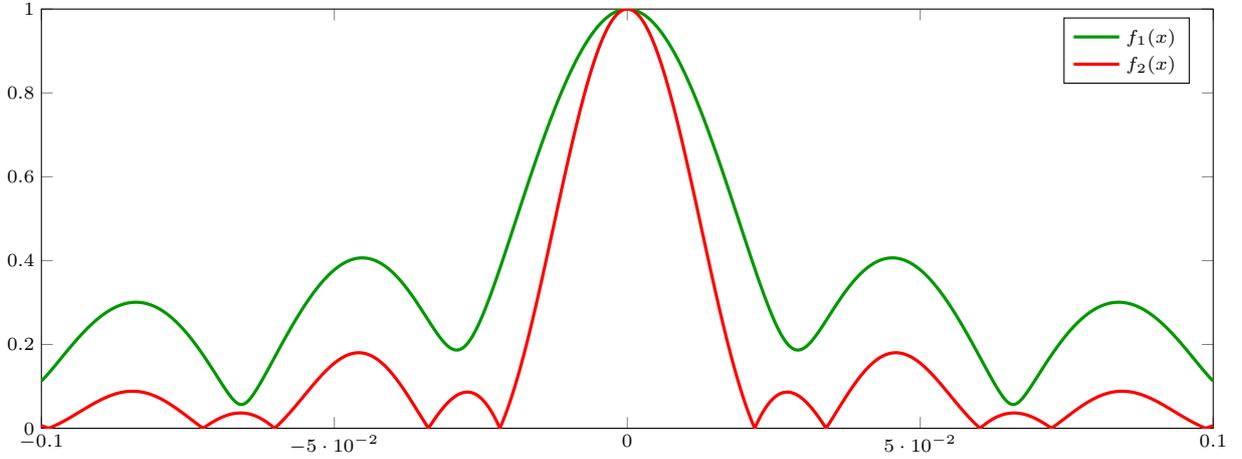
\begin{figure}[h]
\begin{center}
\begin{tikzpicture}[scale=1]
\scriptsize
\begin{axis}
[width=\textwidth,
height=0.42\textwidth,
xmin=-0.1,
xmax=0.1,
ymin=0,
ymax=1,
xtick distance=0.05,
legend cell align={left}]
\addplot[line width=1.2pt,solid,color=green!60!black] %
	table[x=x,y=y1,col sep=comma]{PlotSeries90.csv};
\addlegendentry{\scriptsize$f_1(x)$};
\addplot[line width=1.2pt,solid,color=red] %
	table[x=x,y=y2,col sep=comma]{PlotSeries90.csv};
\addlegendentry{\scriptsize$f_2(x)$};
\end{axis}
\end{tikzpicture}
\caption{\label{PlotSeries}Plots of $f_1(x)$ and $f_2(x)$ at $f=\SI{4}{\giga\hertz}$.}
\end{center}
\end{figure}

\begin{Discussion}[Influence of Bistatic Angle]\label{discussion6}
Based on the identified structure \eqref{Structure_Multiple}, we observe that the imaging performance of $\mathfrak{F}_{\msm}(\mr,C)$ is highly dependent on $\alpha$. Similar to imaging with a single source, a good imaging result is expected when $\alpha$ is small; however, it will be impossible to identify the objects when $\alpha$ is close to $\pi$.
\end{Discussion}

\begin{Discussion}[Influence of Converted Constant]\label{discussion7}
Here, we assume that $\mr=\mr_s\in D_s$ for some $s$. Then, based on \eqref{Structure_Multiple},
\begin{multline*}
\Lambda(\mr_s,\mr')+\Gamma(\mr_s)\\
=\frac{\#[\mathcal{I}(m)]}{8\sqrt{PQ}\pi}\int_D\mathcal{O}(\mr')\rd\mr'+\frac{C\#[\mathcal{J}(m)]}{k}\bigg(J_0(k|\mr_s|)^2+2\sum_{q=1}^{\infty}(-1)^qJ_q(k|\mr_s|)^2\sinc(q\alpha)\bigg).
\end{multline*}
Thus, if $C$ satisfies
\[|C|\ll\left|\frac{k\area(D_s)\#[\mathcal{I}(m)]}{8\pi\#[\mathcal{J}(m)]\sqrt{PQ}}\left(\frac{\eps_s-\epsb}{\epsb\mub}\right)\bigg(J_0(k|\mr_s|)^2+2\sum_{q=1}^{\infty}(-1)^qJ_q(k|\mr_s|)^2\sinc(q\alpha)\bigg)^{-1}\right|,\]
then it will be possible to recognize $D_s$ through the map of $\mathfrak{F}_{\msm}(\mr,C)$. Otherwise, if $C$ satisfies the following inequality for all $s=1,2,\ldots,S$:
\[|C|\gg\left|\frac{k\area(D_s)\#[\mathcal{I}(m)]}{8\pi\#[\mathcal{J}(m)]\sqrt{PQ}}\left(\frac{\eps_s-\epsb}{\epsb\mub}\right)\bigg(J_0(k|\mr_s|)^2+2\sum_{q=1}^{\infty}(-q)^1J_q(k|\mr_s|)^2\sinc(q\alpha)\bigg)^{-1}\right|,\]
then it will be very difficult to recognize the existence, location, and shape of the objects because the map of $\mathfrak{F}_{\msm}(\mr,C)$ will only contain a peak of large magnitude at the origin. 

Based on the above observations, similar to the imaging with a single source, the selection of $C$ to identify $D_s$ is highly dependent on the $\eps_s$ values, the size of $D_s$, $k$, and the total number of transmitting and receiving antennas. Thus, $C=0$ is the best choice for a proper application of $\mathfrak{F}_{\msm}(\mr,C)$.
\end{Discussion}

\section{Results of Numerical Simulations}\label{sec:5}
Here, the results of numerical simulations obtained on the Fresnel dataset \cite{BS} with $f=\SI{4}{\giga\hertz}$ are presented and analyzed to support the theoretical results and elucidate the discovered properties of the DSM. This dataset contains the scattered field data in the presence of two circular objects $D_s$, $s=1,2$ with centers $\mr_1=(-\SI{0.045}{\meter},0)$, $\mr_2=(\SI{0.045}{\meter},\SI{0.010}{\meter})$, the same radii $\alpha_s=\SI{0.015}{\meter}$, and the permittivity values ($\eps_s=(3\pm0.3)\eps_0$). In addition, the transmitters and receivers are positioned on the circles centered at the origin with radii $|P|=\SI{0.72}{\meter}$ and $|Q|=\SI{0.76}{\meter}$, respectively. The range of receivers is restricted from $\alpha$ to $2\pi-\alpha$ with a step size of $\triangle\theta=\SI{5}{\degree}$ based on each direction of the transmitters $\vartheta_m$, and the transmitters are distributed evenly with step sizes of $\SI{10}{\degree}$ from $\SI{0}{\degree}$ to $\triangle\vartheta=\SI{350}{\degree}$. With this setting, the imaging results $\mathfrak{F}_{\dsm}(\mr,m,C)$ with sources $m=1$, $16$, and $31$ (Figure \ref{SimulationSetting} shows the antenna arrangements), and $\mathfrak{F}_{\msm}(\mr,C)$ with every source $m=1,2,\ldots,36$ were produced for each $\mr\in\Omega$, where the imaging region $\Omega$ was selected as the square $[-\SI{0.1}{\meter},\SI{0.1}{\meter}]\times[-\SI{0.1}{\meter},\SI{0.1}{\meter}]$.

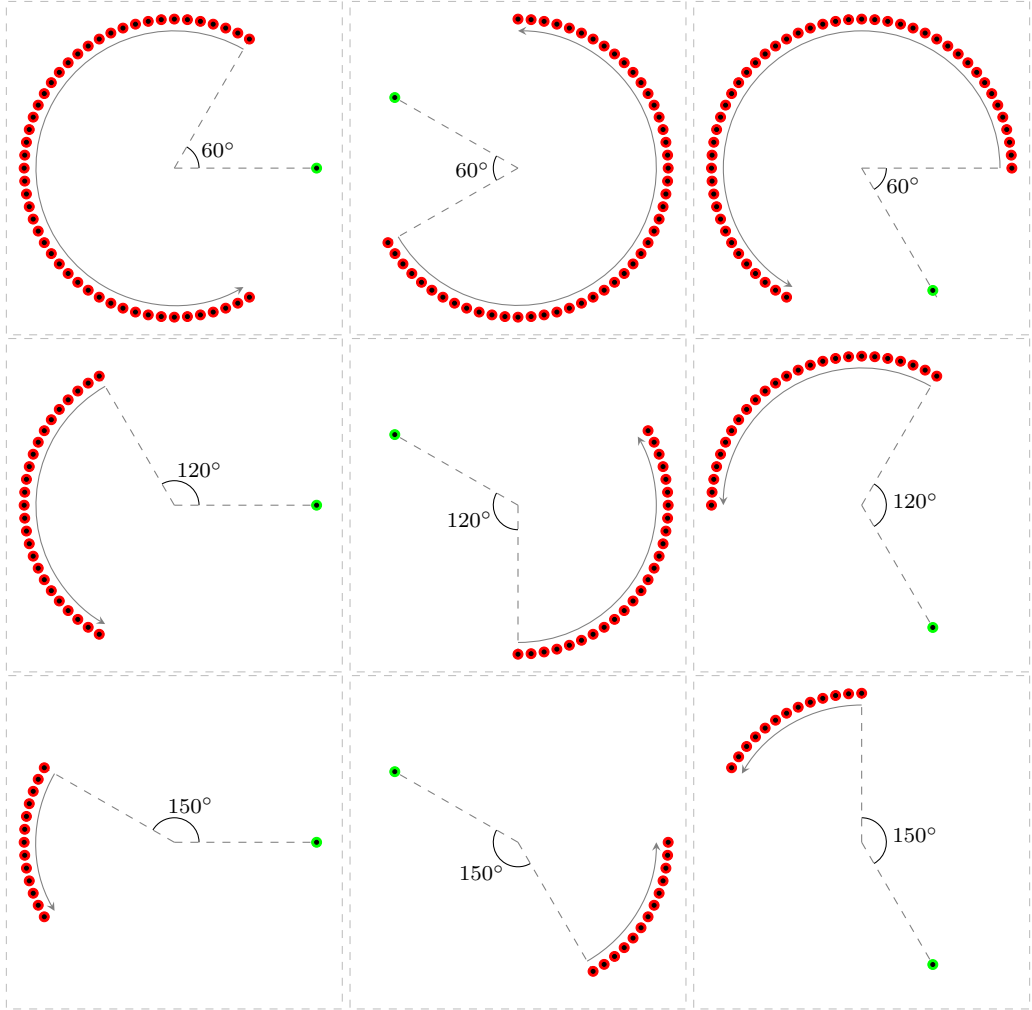
\begin{figure}[h]
\begin{center}
\begin{tikzpicture}[scale=1.3]
\def\RT{0.72*2};
\def\RR{0.76*2};
\def\AL{0.5};
\def\Edge{0.15*2};
\def\BD{0.85*2};

\draw[gray,dashed,-] (0,0) -- (\RT,0);
\draw[gray,dashed,-] (0,0) -- ({\RR*cos(60)},{\RR*sin(60)});

\draw[black,solid,-] (0.25,0) arc (0:60:0.25) node[right,black,xshift=2,yshift=-1] {\footnotesize$\SI{60}{\degree}$};

\draw[green,fill=green] ({\RT*cos(0)},{\RT*sin(0)}) circle (0.05cm);
\draw[black,fill=black] ({\RT*cos(0)},{\RT*sin(0)}) circle (0.02cm);
\foreach \beta in {60,65,...,300}
{\draw[red,fill=red] ({\RR*cos(\beta)},{\RR*sin(\beta)}) circle (0.05cm);
\draw[black,fill=black] ({\RR*cos(\beta)},{\RR*sin(\beta)}) circle (0.02cm);}

\draw[gray,solid,-stealth] ({1.4*cos(60)},{1.4*sin(60)}) arc (60:300:1.4);


\draw[gray!50!white,dashed] (\BD,\BD) -- (\BD,-\BD) -- (-\BD,-\BD) -- (-\BD,\BD) -- cycle;
\end{tikzpicture}
\begin{tikzpicture}[scale=1.3]

\def\RT{0.72*2};
\def\RR{0.76*2};
\def\Edge{0.15*2};
\def\BD{0.85*2};

\draw[gray,dashed,-] (0,0) -- ({\RT*cos(150)},{\RT*sin(150)});
\draw[gray,dashed,-] (0,0) -- ({\RR*cos(210)},{\RR*sin(210)});

\draw[black,solid,-] ({0.25*cos(150)},{0.25*sin(150)}) arc (150:210:0.25) node[left,black,xshift=1,yshift=4] {\footnotesize$\SI{60}{\degree}$};

\draw[green,fill=green] ({\RT*cos(150)},{\RT*sin(150)}) circle (0.05cm);
\draw[black,fill=black] ({\RT*cos(150)},{\RT*sin(150)}) circle (0.02cm);

\foreach \beta in {210,215,...,450}
{\draw[red,fill=red] ({\RR*cos(\beta)},{\RR*sin(\beta)}) circle (0.05cm);
\draw[black,fill=black] ({\RR*cos(\beta)},{\RR*sin(\beta)}) circle (0.02cm);}

\draw[gray,solid,-stealth] ({1.4*cos(210)},{1.4*sin(210)}) arc (210:450:1.4);


\draw[gray!50!white,dashed] (\BD,\BD) -- (\BD,-\BD) -- (-\BD,-\BD) -- (-\BD,\BD) -- cycle;
\end{tikzpicture}
\begin{tikzpicture}[scale=1.3]

\def\RT{0.72*2};
\def\RR{0.76*2};
\def\Edge{0.15*2};
\def\BD{0.85*2};

\draw[gray,dashed,-] (0,0) -- ({\RT*cos(0)},{\RT*sin(0)});
\draw[gray,dashed,-] (0,0) -- ({\RR*cos(300)},{\RR*sin(300)});

\draw[black,solid,-] ({0.25*cos(300)},{0.25*sin(300)}) arc (300:360:0.25) node[below right,black,xshift=-3.5] {\footnotesize$\SI{60}{\degree}$};

\draw[green,fill=green] ({\RT*cos(300)},{\RT*sin(300)}) circle (0.05cm);
\draw[black,fill=black] ({\RT*cos(300)},{\RT*sin(300)}) circle (0.02cm);

\foreach \beta in {0,5,...,240}
{\draw[red,fill=red] ({\RR*cos(\beta)},{\RR*sin(\beta)}) circle (0.05cm);
\draw[black,fill=black] ({\RR*cos(\beta)},{\RR*sin(\beta)}) circle (0.02cm);}

\draw[gray,solid,-stealth] ({1.4*cos(0)},{1.4*sin(0)}) arc (0:240:1.4);


\draw[gray!50!white,dashed] (\BD,\BD) -- (\BD,-\BD) -- (-\BD,-\BD) -- (-\BD,\BD) -- cycle;
\end{tikzpicture}\\
\begin{tikzpicture}[scale=1.3]
\def\RT{0.72*2};
\def\RR{0.76*2};
\def\Edge{0.15*2};
\def\BD{0.85*2};

\draw[gray,dashed,-] (0,0) -- (\RT,0);
\draw[gray,dashed,-] (0,0) -- ({\RR*cos(120)},{\RR*sin(120)});

\draw[black,solid,-] (0.25,0) arc (0:120:0.25) node[above right,black,xshift=2,yshift=-1] {\footnotesize$\SI{120}{\degree}$};

\draw[green,fill=green] ({\RT*cos(0)},{\RT*sin(0)}) circle (0.05cm);
\draw[black,fill=black] ({\RT*cos(0)},{\RT*sin(0)}) circle (0.02cm);
\foreach \beta in {120,125,...,240}
{\draw[red,fill=red] ({\RR*cos(\beta)},{\RR*sin(\beta)}) circle (0.05cm);
\draw[black,fill=black] ({\RR*cos(\beta)},{\RR*sin(\beta)}) circle (0.02cm);}

\draw[gray,solid,-stealth] ({1.4*cos(120)},{1.4*sin(120)}) arc (120:240:1.4);


\draw[gray!50!white,dashed] (\BD,\BD) -- (\BD,-\BD) -- (-\BD,-\BD) -- (-\BD,\BD) -- cycle;
\end{tikzpicture}
\begin{tikzpicture}[scale=1.3]

\def\RT{0.72*2};
\def\RR{0.76*2};
\def\Edge{0.15*2};
\def\BD{0.85*2};

\draw[gray,dashed,-] (0,0) -- ({\RT*cos(150)},{\RT*sin(150)});
\draw[gray,dashed,-] (0,0) -- ({\RR*cos(270)},{\RR*sin(270)});

\draw[black,solid,-] ({0.25*cos(150)},{0.25*sin(150)}) arc (150:270:0.25) node[left,black,xshift=-6,yshift=4] {\footnotesize$\SI{120}{\degree}$};

\draw[green,fill=green] ({\RT*cos(150)},{\RT*sin(150)}) circle (0.05cm);
\draw[black,fill=black] ({\RT*cos(150)},{\RT*sin(150)}) circle (0.02cm);

\foreach \beta in {270,275,...,390}
{\draw[red,fill=red] ({\RR*cos(\beta)},{\RR*sin(\beta)}) circle (0.05cm);
\draw[black,fill=black] ({\RR*cos(\beta)},{\RR*sin(\beta)}) circle (0.02cm);}

\draw[gray,solid,-stealth] ({1.4*cos(270)},{1.4*sin(270)}) arc (270:390:1.4);


\draw[gray!50!white,dashed] (\BD,\BD) -- (\BD,-\BD) -- (-\BD,-\BD) -- (-\BD,\BD) -- cycle;
\end{tikzpicture}
\begin{tikzpicture}[scale=1.3]

\def\RT{0.72*2};
\def\RR{0.76*2};
\def\Edge{0.15*2};
\def\BD{0.85*2};

\draw[gray,dashed,-] (0,0) -- ({\RT*cos(300)},{\RT*sin(300)});
\draw[gray,dashed,-] (0,0) -- ({\RR*cos(60)},{\RR*sin(60)});

\draw[black,solid,-] ({0.25*cos(300)},{0.25*sin(300)}) arc (300:420:0.25) node[below right,black,xshift=3.5] {\footnotesize$\SI{120}{\degree}$};

\draw[green,fill=green] ({\RT*cos(300)},{\RT*sin(300)}) circle (0.05cm);
\draw[black,fill=black] ({\RT*cos(300)},{\RT*sin(300)}) circle (0.02cm);

\foreach \beta in {60,65,...,180}
{\draw[red,fill=red] ({\RR*cos(\beta)},{\RR*sin(\beta)}) circle (0.05cm);
\draw[black,fill=black] ({\RR*cos(\beta)},{\RR*sin(\beta)}) circle (0.02cm);}

\draw[gray,solid,-stealth] ({1.4*cos(60)},{1.4*sin(60)}) arc (60:180:1.4);


\draw[gray!50!white,dashed] (\BD,\BD) -- (\BD,-\BD) -- (-\BD,-\BD) -- (-\BD,\BD) -- cycle;
\end{tikzpicture}\\
\begin{tikzpicture}[scale=1.3]
\def\RT{0.72*2};
\def\RR{0.76*2};
\def\Edge{0.15*2};
\def\BD{0.85*2};

\draw[gray,dashed,-] (0,0) -- (\RT,0);
\draw[gray,dashed,-] (0,0) -- ({\RR*cos(150)},{\RR*sin(150)});

\draw[black,solid,-] (0.25,0) arc (0:150:0.25) node[above right,black,xshift=2,yshift=3] {\footnotesize$\SI{150}{\degree}$};

\draw[green,fill=green] ({\RT*cos(0)},{\RT*sin(0)}) circle (0.05cm);
\draw[black,fill=black] ({\RT*cos(0)},{\RT*sin(0)}) circle (0.02cm);
\foreach \beta in {150,155,...,210}
{\draw[red,fill=red] ({\RR*cos(\beta)},{\RR*sin(\beta)}) circle (0.05cm);
\draw[black,fill=black] ({\RR*cos(\beta)},{\RR*sin(\beta)}) circle (0.02cm);}

\draw[gray,solid,-stealth] ({1.4*cos(150)},{1.4*sin(150)}) arc (150:210:1.4);


\draw[gray!50!white,dashed] (\BD,\BD) -- (\BD,-\BD) -- (-\BD,-\BD) -- (-\BD,\BD) -- cycle;
\end{tikzpicture}
\begin{tikzpicture}[scale=1.3]

\def\RT{0.72*2};
\def\RR{0.76*2};
\def\Edge{0.15*2};
\def\BD{0.85*2};

\draw[gray,dashed,-] (0,0) -- ({\RT*cos(150)},{\RT*sin(150)});
\draw[gray,dashed,-] (0,0) -- ({\RR*cos(300)},{\RR*sin(300)});

\draw[black,solid,-] ({0.25*cos(150)},{0.25*sin(150)}) arc (150:300:0.25) node[below left,black,xshift=-6,yshift=3] {\footnotesize$\SI{150}{\degree}$};

\draw[green,fill=green] ({\RT*cos(150)},{\RT*sin(150)}) circle (0.05cm);
\draw[black,fill=black] ({\RT*cos(150)},{\RT*sin(150)}) circle (0.02cm);

\foreach \beta in {300,305,...,360}
{\draw[red,fill=red] ({\RR*cos(\beta)},{\RR*sin(\beta)}) circle (0.05cm);
\draw[black,fill=black] ({\RR*cos(\beta)},{\RR*sin(\beta)}) circle (0.02cm);}

\draw[gray,solid,-stealth] ({1.4*cos(300)},{1.4*sin(300)}) arc (300:360:1.4);


\draw[gray!50!white,dashed] (\BD,\BD) -- (\BD,-\BD) -- (-\BD,-\BD) -- (-\BD,\BD) -- cycle;
\end{tikzpicture}
\begin{tikzpicture}[scale=1.3]

\def\RT{0.72*2};
\def\RR{0.76*2};
\def\Edge{0.15*2};
\def\BD{0.85*2};

\draw[gray,dashed,-] (0,0) -- ({\RT*cos(300)},{\RT*sin(300)});
\draw[gray,dashed,-] (0,0) -- ({\RR*cos(90)},{\RR*sin(90)});

\draw[black,solid,-] ({0.25*cos(300)},{0.25*sin(300)}) arc (300:450:0.25) node[below right,black,xshift=8] {\footnotesize$\SI{150}{\degree}$};

\draw[green,fill=green] ({\RT*cos(300)},{\RT*sin(300)}) circle (0.05cm);
\draw[black,fill=black] ({\RT*cos(300)},{\RT*sin(300)}) circle (0.02cm);

\foreach \beta in {90,95,...,150}
{\draw[red,fill=red] ({\RR*cos(\beta)},{\RR*sin(\beta)}) circle (0.05cm);
\draw[black,fill=black] ({\RR*cos(\beta)},{\RR*sin(\beta)}) circle (0.02cm);}

\draw[gray,solid,-stealth] ({1.4*cos(90)},{1.4*sin(90)}) arc (90:150:1.4);


\draw[gray!50!white,dashed] (\BD,\BD) -- (\BD,-\BD) -- (-\BD,-\BD) -- (-\BD,\BD) -- cycle;
\end{tikzpicture}
\caption{\label{SimulationSetting}Antenna arrangements with various bistatic angles $\alpha=\SI{60}{\degree}$, $\SI{120}{\degree}$, and $\SI{150}{\degree}$ corresponding to the transmitter $\mathcal{P}_1$ (left column), $\mathcal{P}_{16}$ (middle column), and $\mathcal{P}_{31}$ (right column).}
\end{center}
\end{figure}

\begin{Example}[Imaging With Zero Constant]\label{Ex1}
Figure \ref{Fig1} shows maps of $\mathfrak{F}_{\dsm}(\mr,m,0)$ with $\alpha=\SI{60}{\degree}$ and various $m=1$, $16$, and $31$. Based on the results, it is possible to recognize the existence, nearly accurate location, and shape of $D_s$ when $m=1$. However, it is very difficult to recognize the existence of $D_1$ and $D_2$ due to the appearance of peaks of large magnitudes when $m=16$. Note that the existence of $D_2$ can be recognized; however, it is very difficult to recognize $D_1$ when $m=31$. Thus, the imaging performance of $\mathfrak{F}_{\dsm}(\mr,m,0)$ is highly dependent on $m$, which supports Discussion \ref{discussion1}.
\end{Example}

\begin{figure}[h]
\includegraphics[width=.33\textwidth]{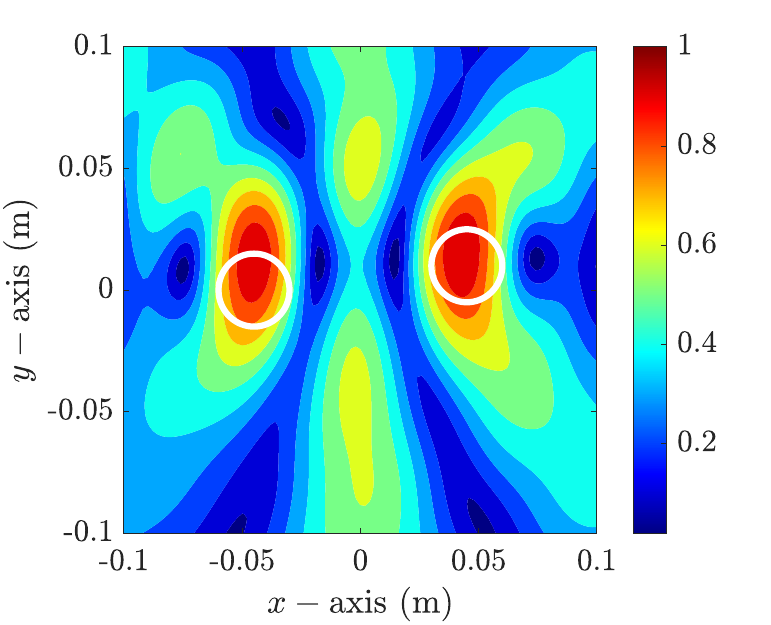}\hfill
\includegraphics[width=.33\textwidth]{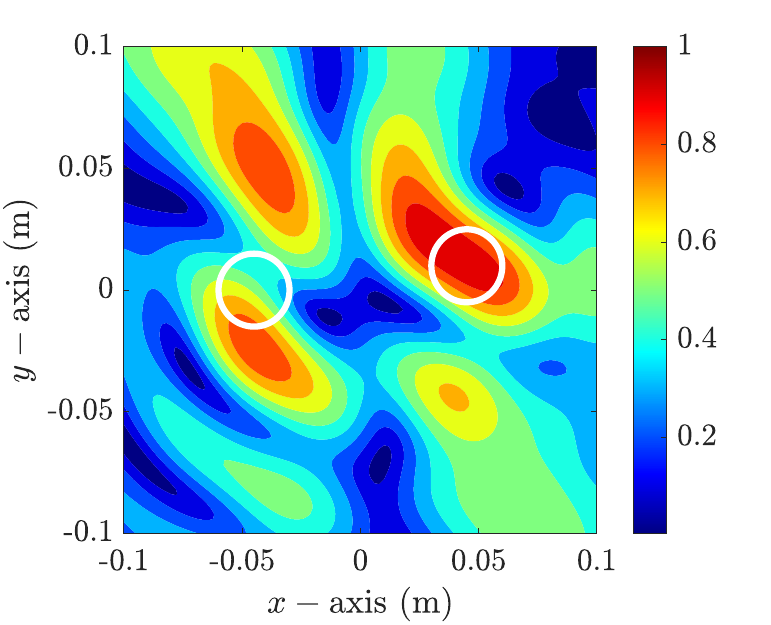}\hfill
\includegraphics[width=.33\textwidth]{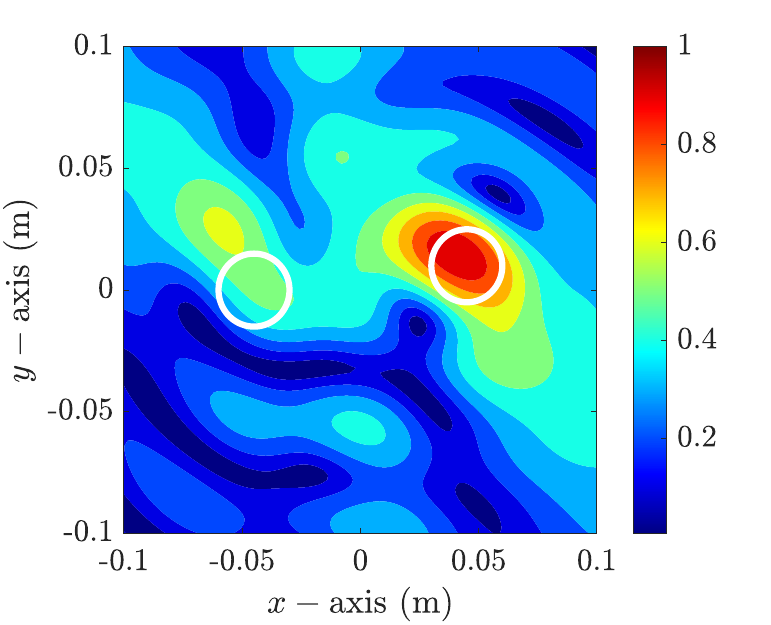}
\caption{\label{Fig1}(Example \ref{Ex1}) Maps of $\mathfrak{F}_{\dsm}(\mr,m,0)$ with $\alpha=\SI{60}{\degree}$ for $m=1$ (left), $m=16$ (middle), and $m=31$ (right).}
\end{figure}

\begin{Example}[Imaging With Nonzero Constant]\label{Ex2}
Figure \ref{Fig2-1} shows maps of $\mathfrak{F}_{\dsm}(\mr,m,0.1)$ with $\alpha=\SI{60}{\degree}$ and various $m$. Similar to Example \ref{Ex1}, here, the existence of $D_1$ and $D_2$ can be recognized; however, more artifacts are included when $m=1$. In addition, poor imaging results were obtained for $m=16$ and $m=31$.

Figure \ref{Fig2-2} shows maps of $\mathfrak{F}_{\dsm}(\mr,m,0.5i)$ with $\alpha=\SI{60}{\degree}$ and various $m$. Compared with the results shown in Figure \ref{Fig2-1}, in this case, it is impossible to recognize both $D_1$ and $D_2$ due to the appearance of several artifacts with large magnitudes. Based on the imaging results $\mathfrak{F}_{\dsm}(\mr,m,2)$ obtained with $\alpha=\SI{60}{\degree}$ and various $m$, it appears to be impossible to recognize $D_1$ and $D_2$ (Figure \ref{Fig2-3}). Thus, based on Discussion \ref{discussion3}, converting unmeasurable scattered field data into the zero constant (i.e., selecting $C=0$) is the best choice.
\end{Example}

\begin{figure}[h]
\includegraphics[width=.33\textwidth]{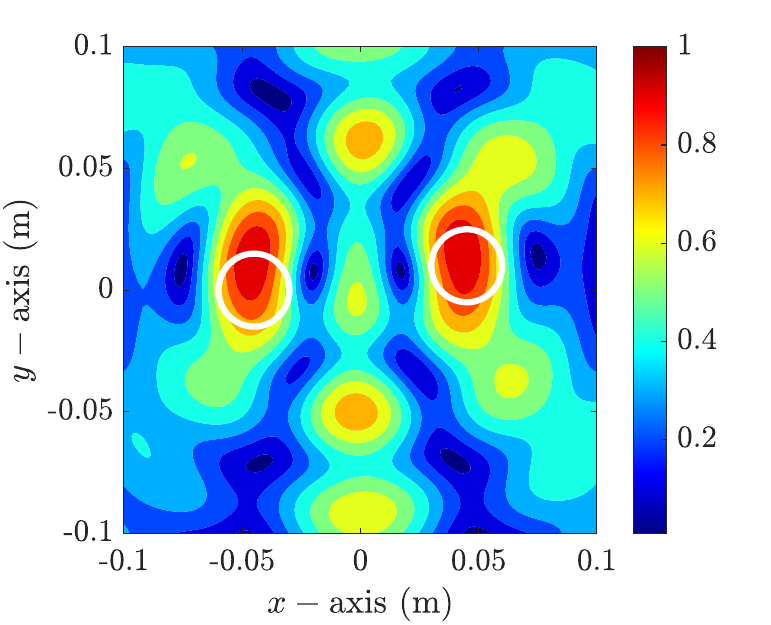}\hfill
\includegraphics[width=.33\textwidth]{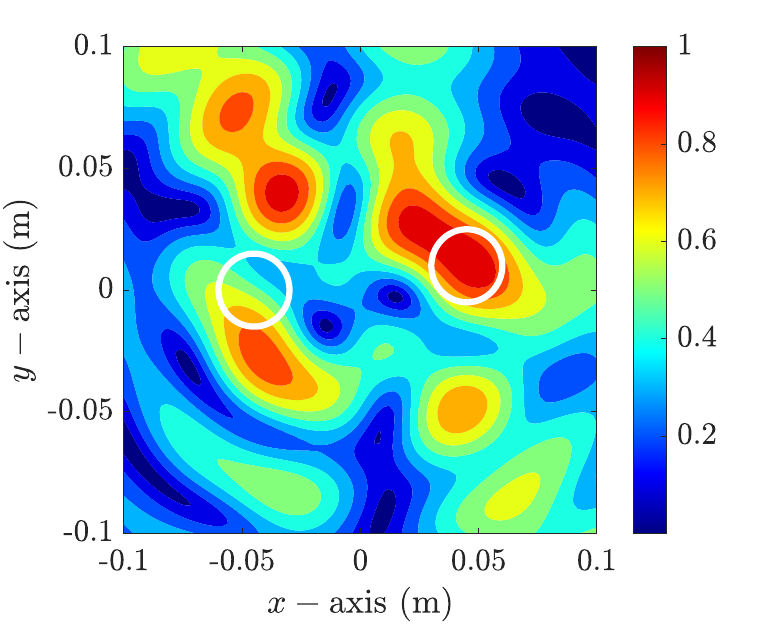}\hfill
\includegraphics[width=.33\textwidth]{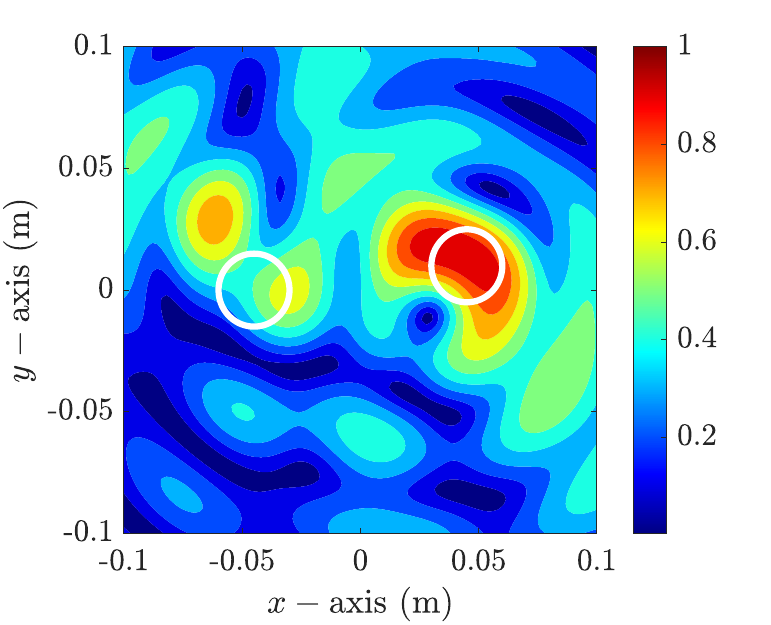}
\caption{\label{Fig2-1}(Example \ref{Ex2}) Maps of $\mathfrak{F}_{\dsm}(\mr,m,0.1)$ with $\alpha=\SI{60}{\degree}$ for $m=1$ (left), $m=16$ (middle), and $m=31$ (right).}
\end{figure}

\begin{figure}[h]
\includegraphics[width=.33\textwidth]{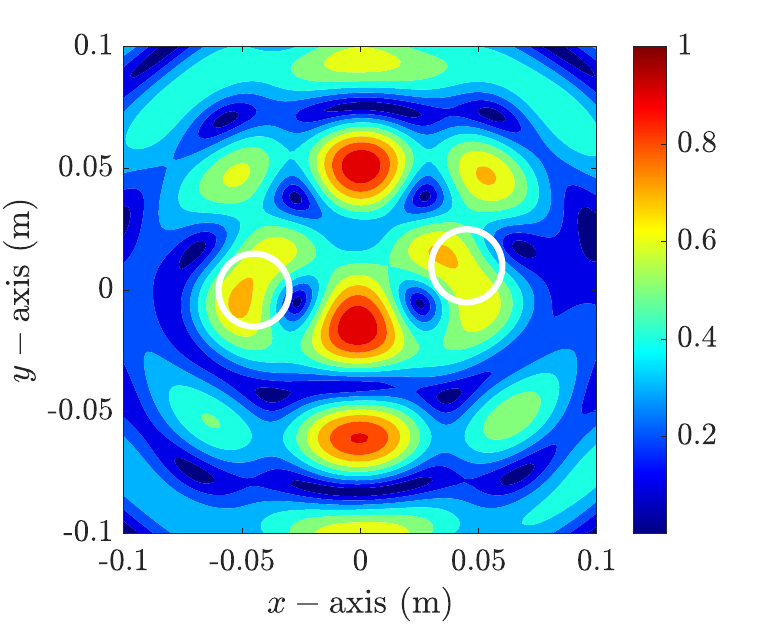}\hfill
\includegraphics[width=.33\textwidth]{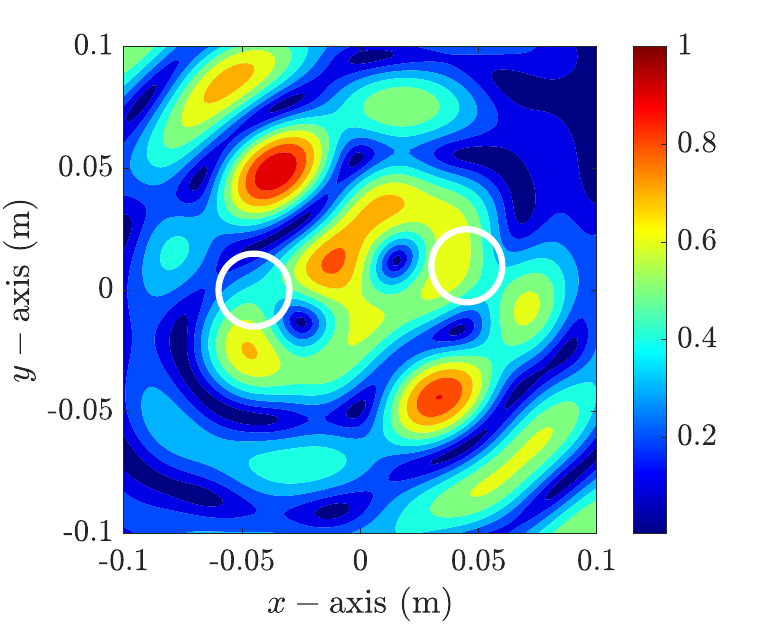}\hfill
\includegraphics[width=.33\textwidth]{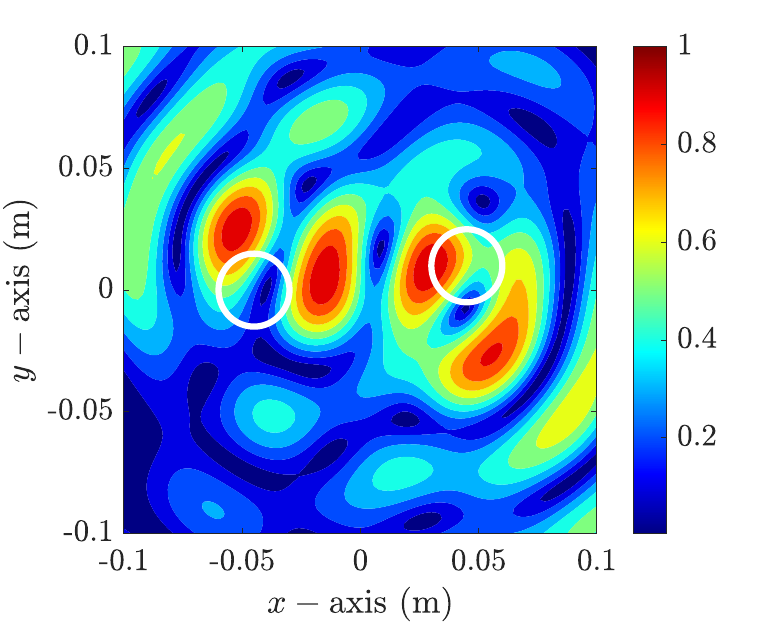}
\caption{\label{Fig2-2}(Example \ref{Ex2}) Maps of $\mathfrak{F}_{\dsm}(\mr,m,0.5i)$ with $\alpha=\SI{60}{\degree}$ for $m=1$ (left), $m=16$ (middle), and $m=31$ (right).}
\end{figure}

\begin{figure}[h]
\includegraphics[width=.33\textwidth]{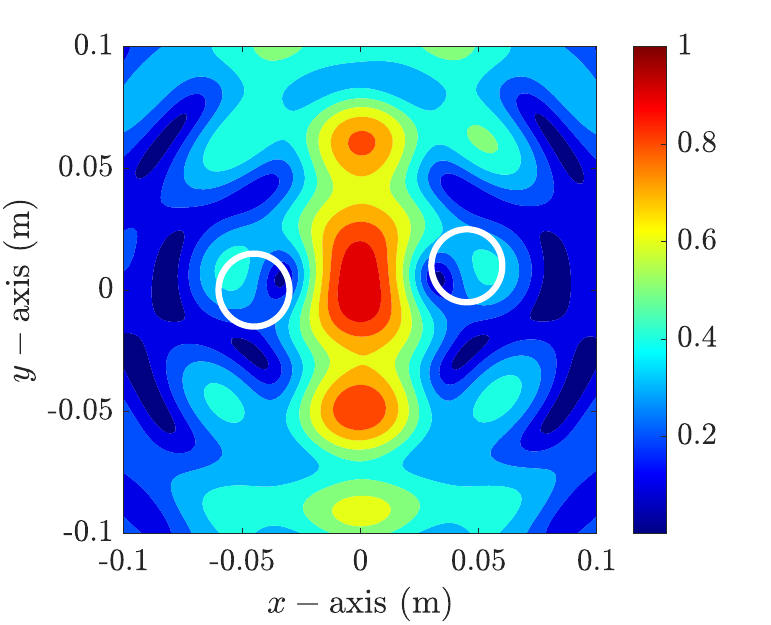}\hfill
\includegraphics[width=.33\textwidth]{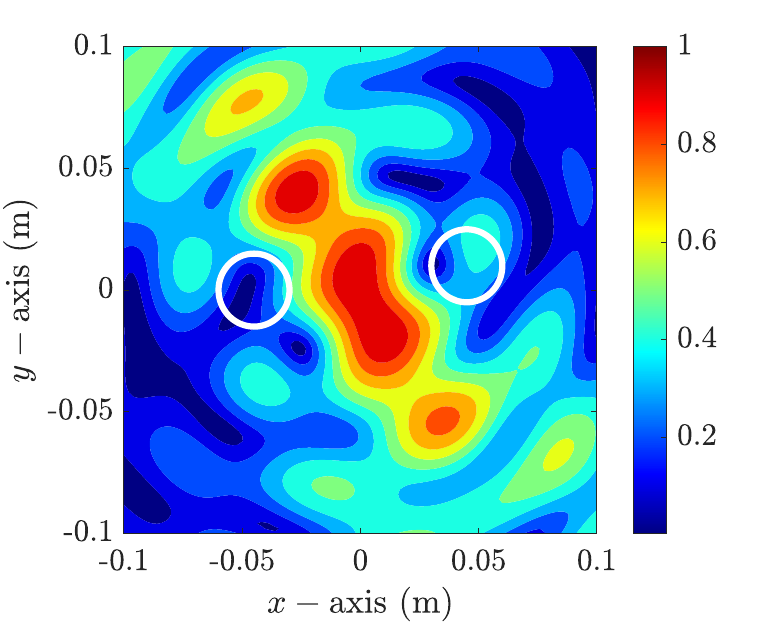}\hfill
\includegraphics[width=.33\textwidth]{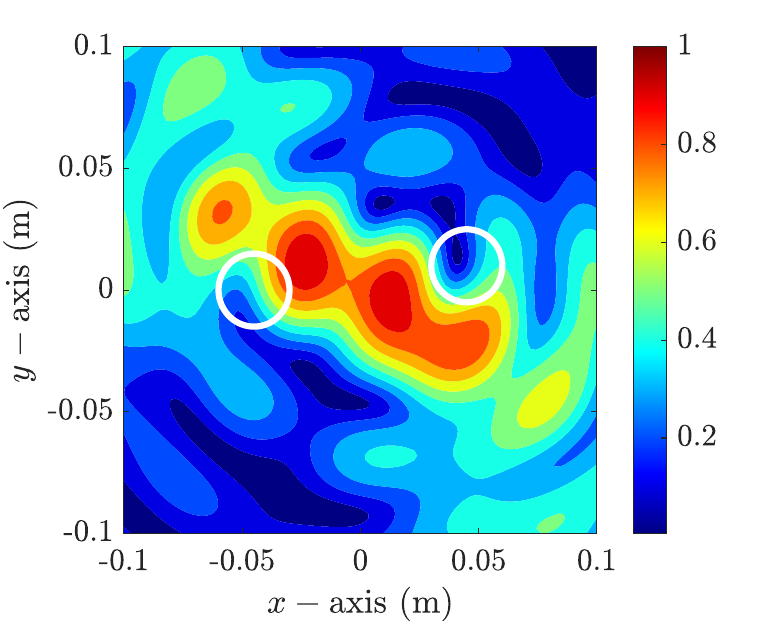}
\caption{\label{Fig2-3}(Example \ref{Ex2}) Maps of $\mathfrak{F}_{\dsm}(\mr,m,2)$ with $\alpha=\SI{60}{\degree}$ for $m=1$ (left), $m=16$ (middle), and $m=31$ (right).}
\end{figure}

\begin{Example}[Imaging With Various Bistatic Angles]\label{Ex3}
Figure \ref{Fig3-1} shows maps of $\mathfrak{F}_{\dsm}(\mr,m,0)$ with $\alpha=\SI{120}{\degree}$ and various $m$. In contrast to the result obtained for Example \ref{Ex1}, here, it is very difficult to identify the outline shape of the objects. In addition, if $C=0.2$ is selected, a peak of large magnitude appears at the origin; thus, it is very difficult to distinguish the objects and artifact when $m=1$ (Figure \ref{Fig3-2}). Furthermore, it is impossible to recognize the objects when $m=16$ and $31$.

Figure \ref{Fig3-3} shows maps of $\mathfrak{F}_{\dsm}(\mr,m,0)$ with $\alpha=\SI{150}{\degree}$ and various $m$. In contrast to the previous results, it is impossible to recognize the existence of the objects because peaks of large magnitudes do not appear at the objects. In addition, if $\alpha=\SI{180}{\degree}$, as observed in Discussion \ref{discussion2}, nothing can be recognized because there is no peak of large magnitude as shown in Figure \ref{Fig3-4}.
\end{Example}

\begin{figure}[h]
\includegraphics[width=.33\textwidth]{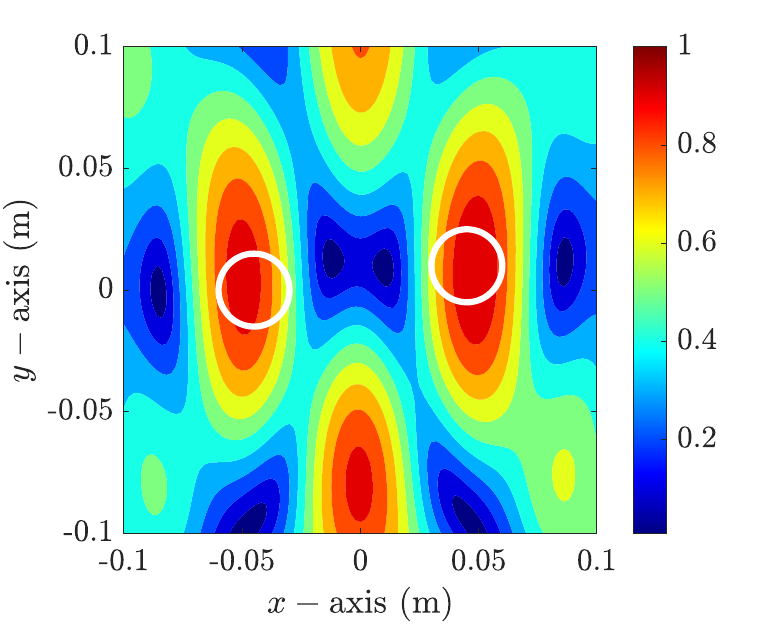}\hfill
\includegraphics[width=.33\textwidth]{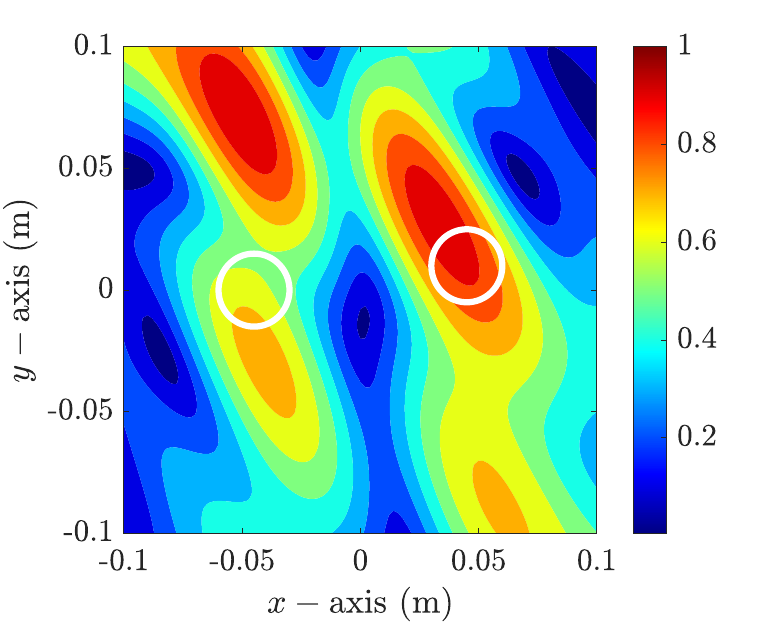}\hfill
\includegraphics[width=.33\textwidth]{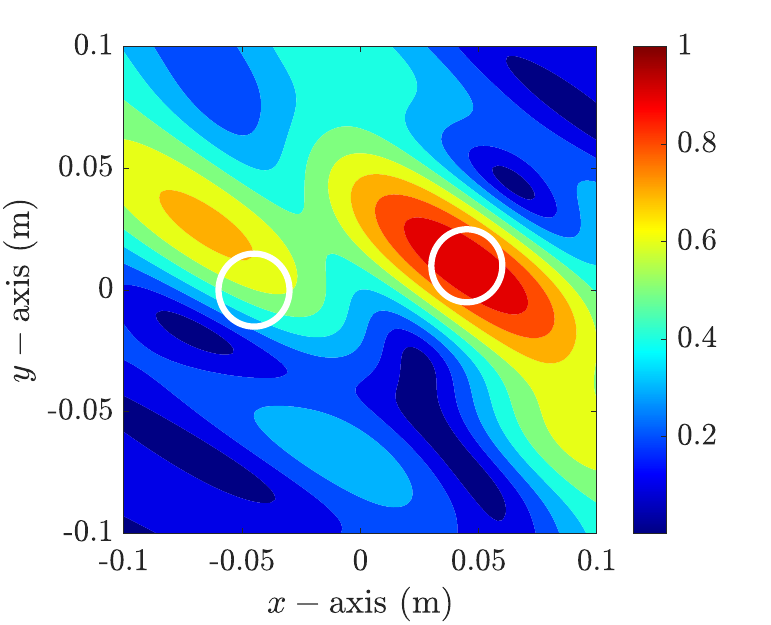}
\caption{\label{Fig3-1}(Example \ref{Ex3}) Maps of $\mathfrak{F}_{\dsm}(\mr,m,0)$ with $\alpha=\SI{120}{\degree}$ for $m=1$ (left), $m=16$ (middle), and $m=31$ (right).}
\end{figure}

\begin{figure}[h]
\includegraphics[width=.33\textwidth]{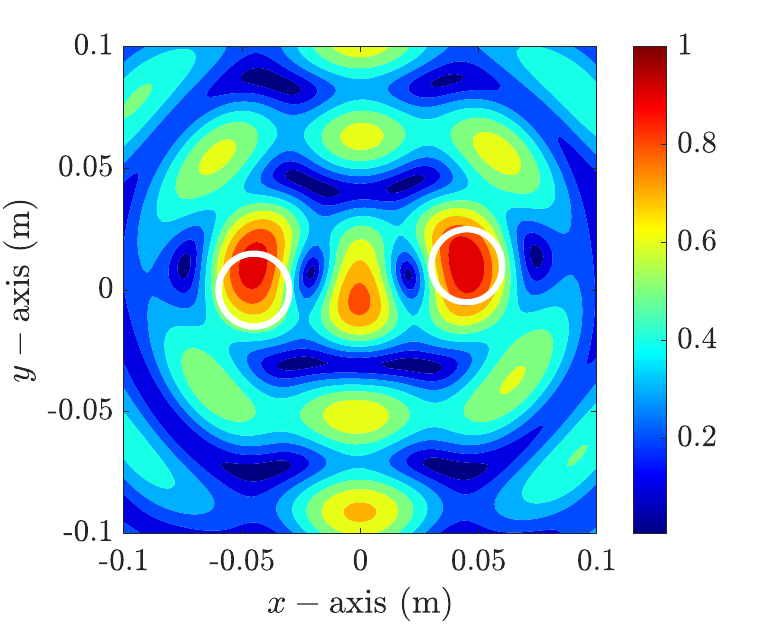}\hfill
\includegraphics[width=.33\textwidth]{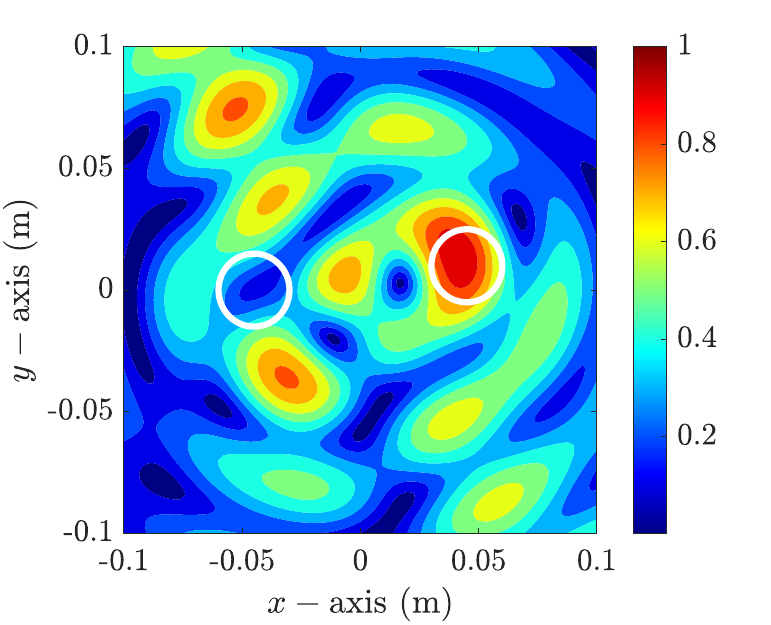}\hfill
\includegraphics[width=.33\textwidth]{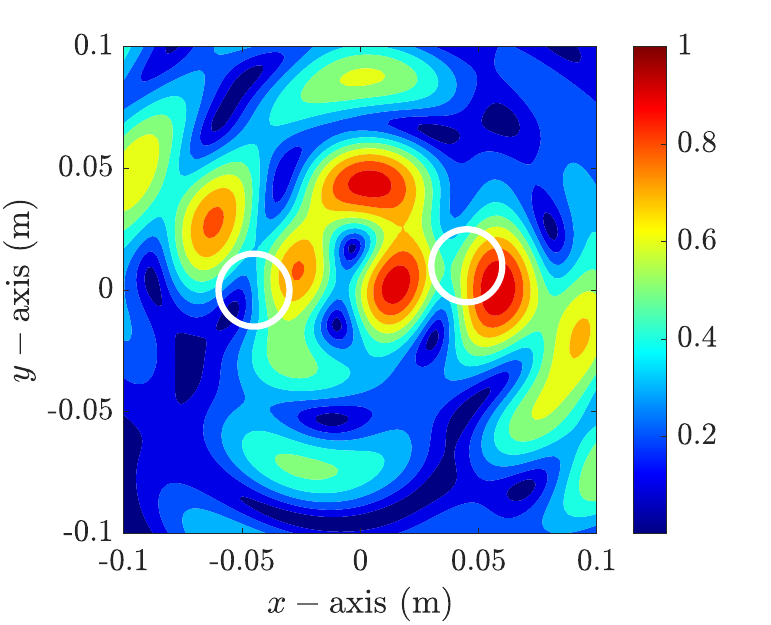}
\caption{\label{Fig3-2}(Example \ref{Ex3}) Maps of $\mathfrak{F}_{\dsm}(\mr,m,0.2)$ with $\alpha=\SI{120}{\degree}$ for $m=1$ (left), $m=16$ (middle), and $m=31$ (right).}
\end{figure}

\begin{figure}[h]
\includegraphics[width=.33\textwidth]{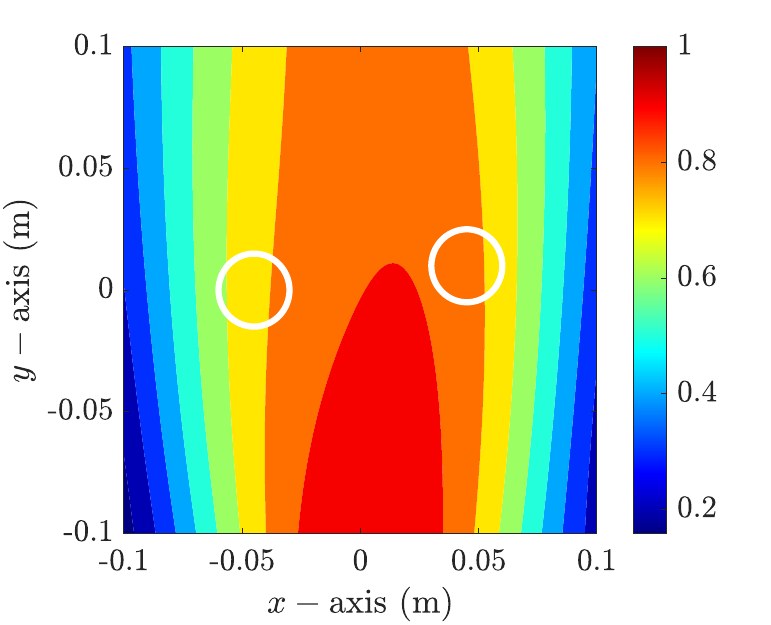}\hfill
\includegraphics[width=.33\textwidth]{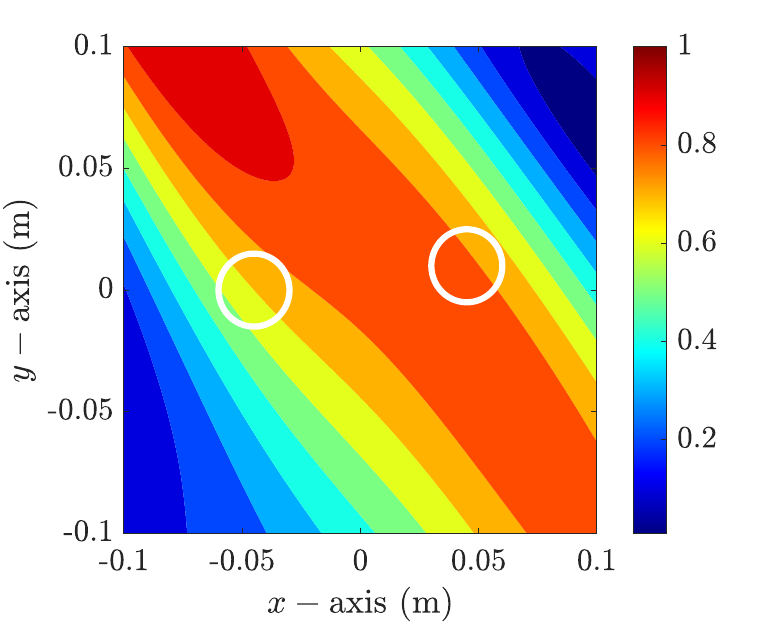}\hfill
\includegraphics[width=.33\textwidth]{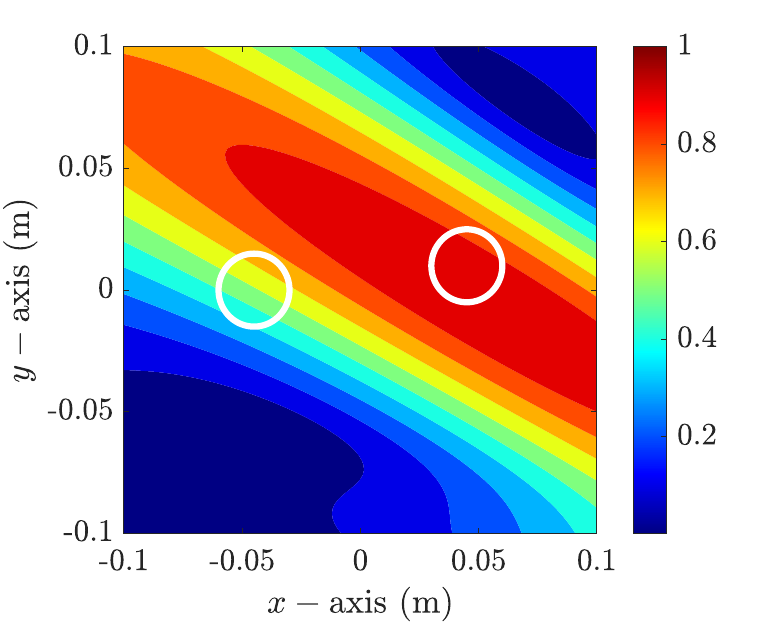}
\caption{\label{Fig3-3}(Example \ref{Ex3}) Maps of $\mathfrak{F}_{\dsm}(\mr,m,0)$ with $\alpha=\SI{150}{\degree}$ for $m=1$ (left), $m=16$ (middle), and $m=31$ (right).}
\end{figure}

\begin{figure}[h]
\includegraphics[width=.33\textwidth]{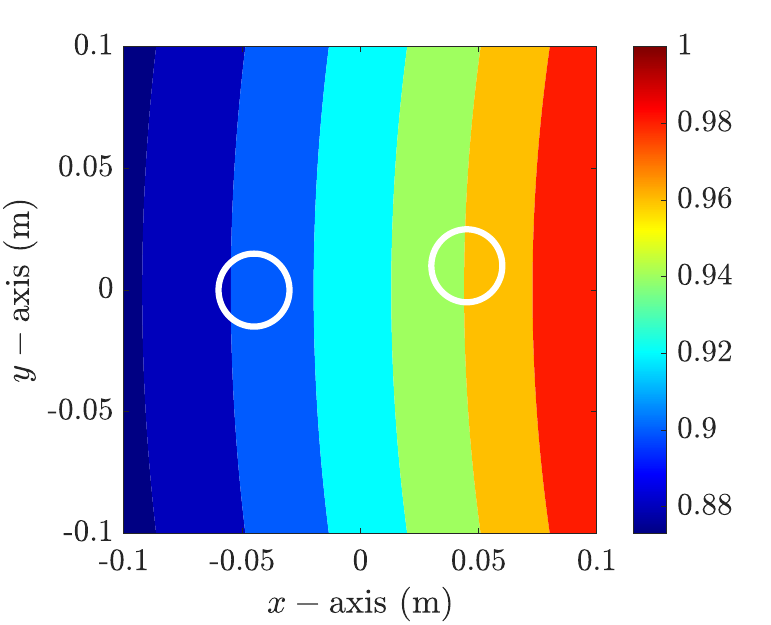}\hfill
\includegraphics[width=.33\textwidth]{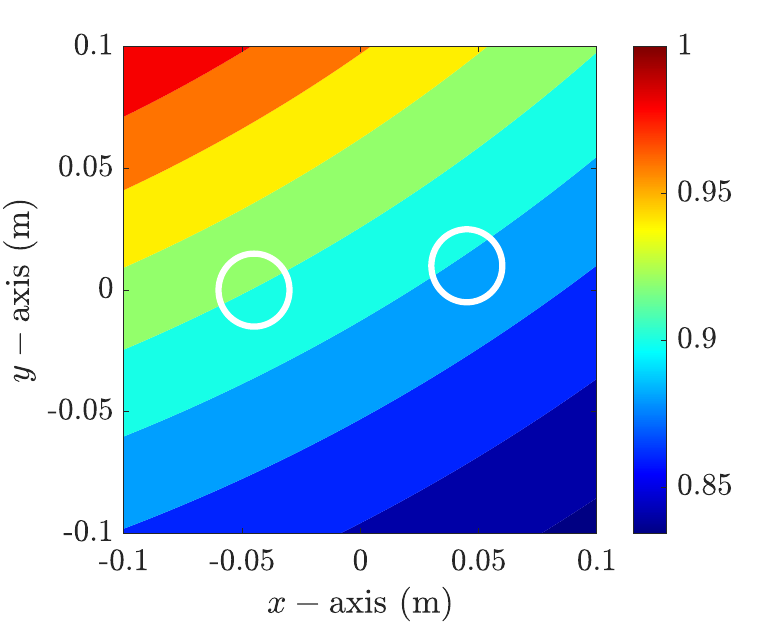}\hfill
\includegraphics[width=.33\textwidth]{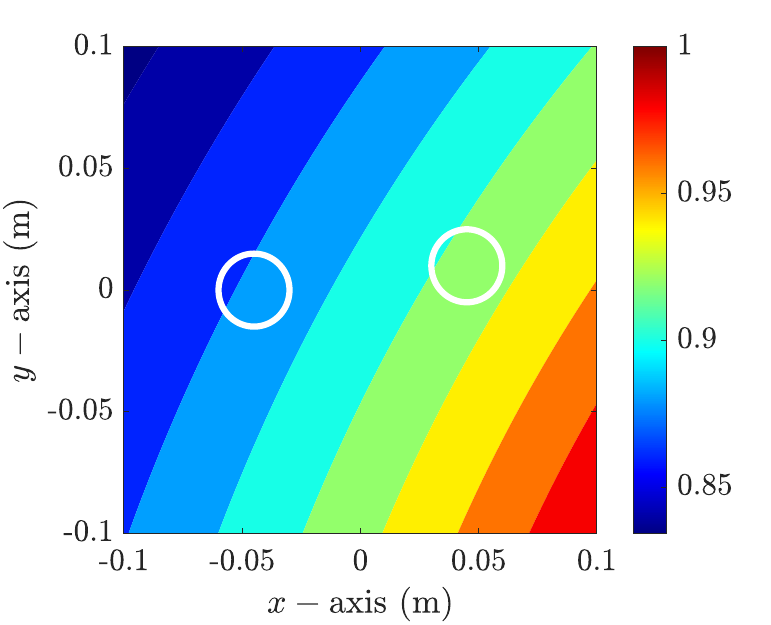}
\caption{\label{Fig3-4}(Example \ref{Ex3}) Maps of $\mathfrak{F}_{\dsm}(\mr,m,0)$ with $\alpha=\SI{180}{\degree}$ for $m=1$ (left), $m=16$ (middle), and $m=31$ (right).}
\end{figure}

\begin{Example}[Imaging With Multiple Sources]\label{Ex4}
Figure \ref{Fig4} shows maps of $\mathfrak{F}_{\msm}(\mr,0)$ with various bistatic angle $\alpha$ values. Here, opposite to Example \ref{Ex3}, it is possible to recognize the existence and approximate shape of both $D_1$ and $D_2$ for $\alpha=\SI{120}{\degree}$. However, it remains impossible to retrieve the objects when $\alpha=\SI{150}{\degree}$ and $\SI{180}{\degree}$. Thus, we conclude that the DSM with multiple sources improves the imaging performance, but it is impossible to retrieve the objects when $\alpha$ approaches $\SI{180}{\degree}$.
\end{Example}

\begin{figure}[h]
\includegraphics[width=.33\textwidth]{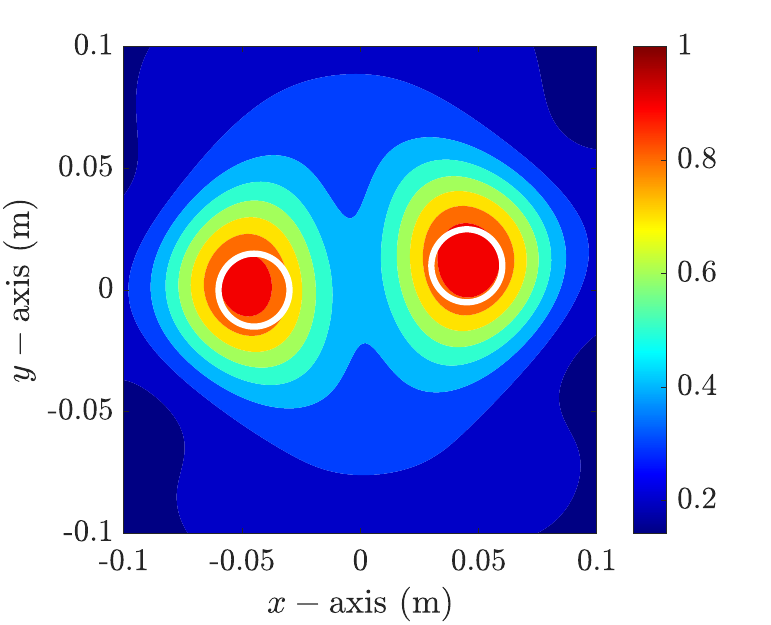}\hfill
\includegraphics[width=.33\textwidth]{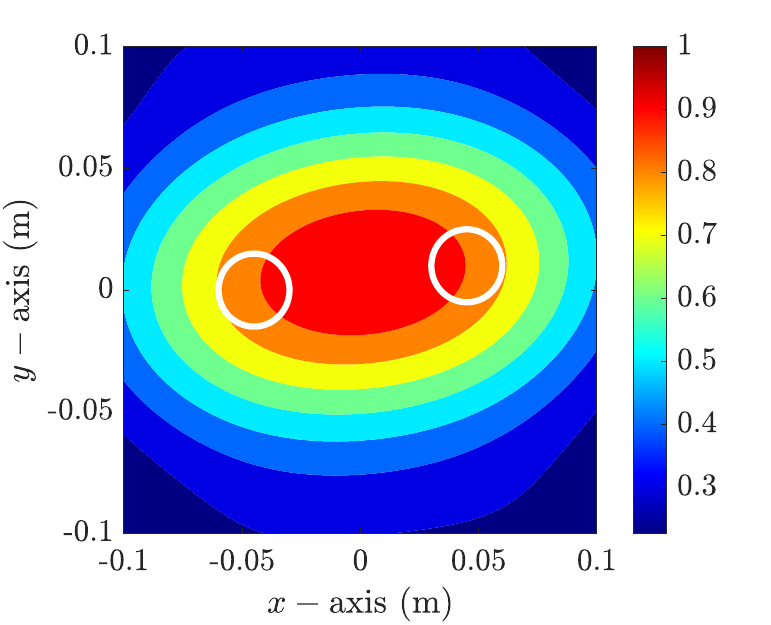}\hfill
\includegraphics[width=.33\textwidth]{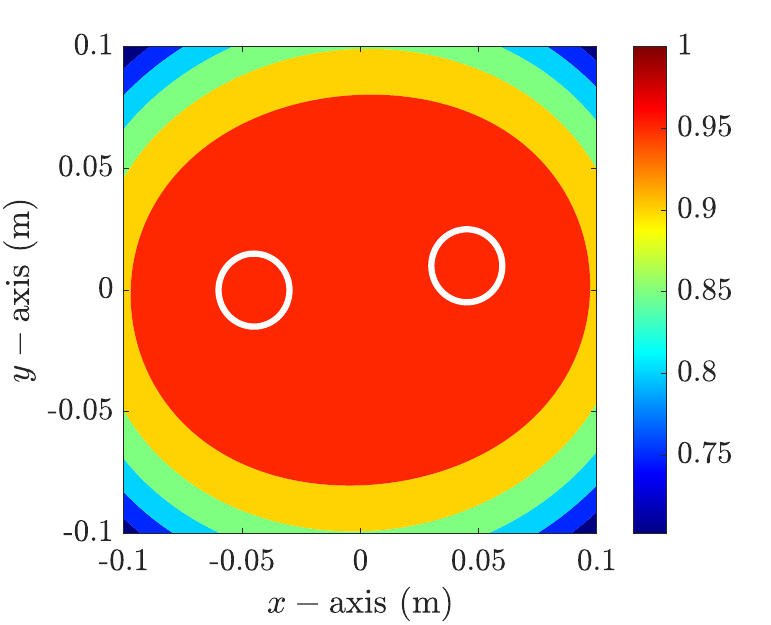}
\caption{\label{Fig4}(Example \ref{Ex4}) Maps of $\mathfrak{F}_{\msm}(\mr,0)$ with $\alpha=\SI{120}{\degree}$ (left), $\alpha=\SI{150}{\degree}$ (middle), and $\alpha=\SI{180}{\degree}$ (right).}
\end{figure}

\begin{Example}[Imaging With Multiple Sources with Nonzero Constant]\label{Ex5}
For the final example, we consider the imaging results with various constant $C$ with bistatic angle $\alpha=\SI{120}{\degree}$. Based on the Figure \ref{Fig5}, it is possible to recognize the existence and location of the objects but their outline shapes cannot be identified when $C=0.1$. Notice that since the values of $\mathfrak{F}_{\msm}(\mr,0.3i)$ are significantly large at the origin and its neighborhood, it seems very difficult to recognize the presence of objects. When a larger $C$ value ($C=0.5$) is applied, since the maximum value of $\mathfrak{F}_{\msm}(\mr,0.5)$ appears only at the origin, it is impossible to recognize the presence of the objects.
\end{Example}

\begin{figure}[h]
\includegraphics[width=.33\textwidth]{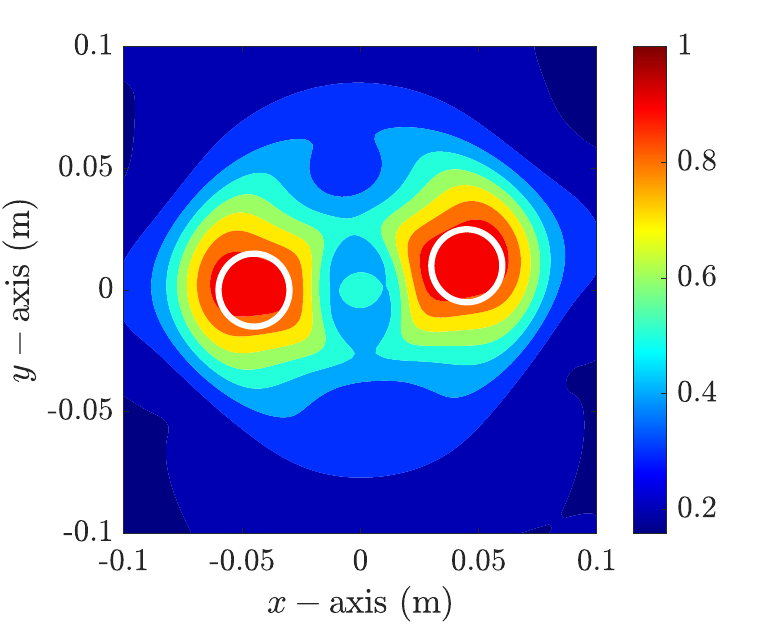}\hfill
\includegraphics[width=.33\textwidth]{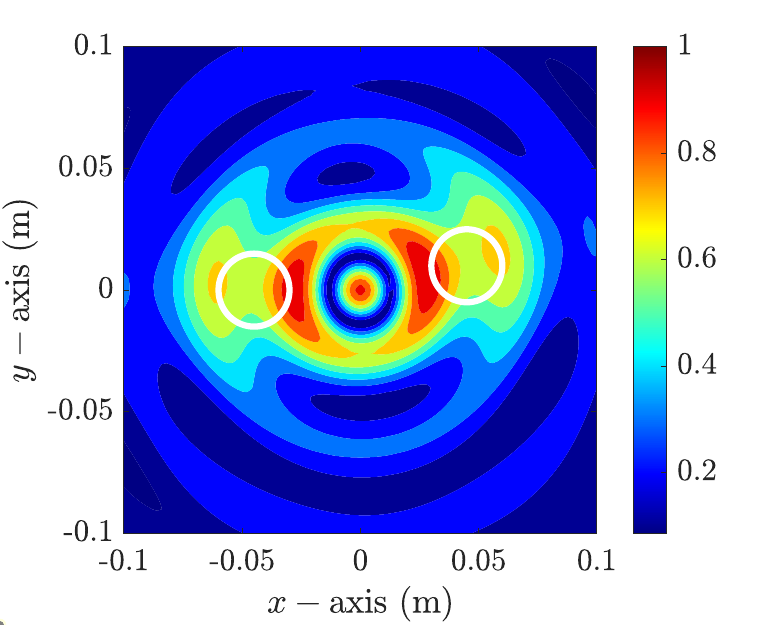}\hfill
\includegraphics[width=.33\textwidth]{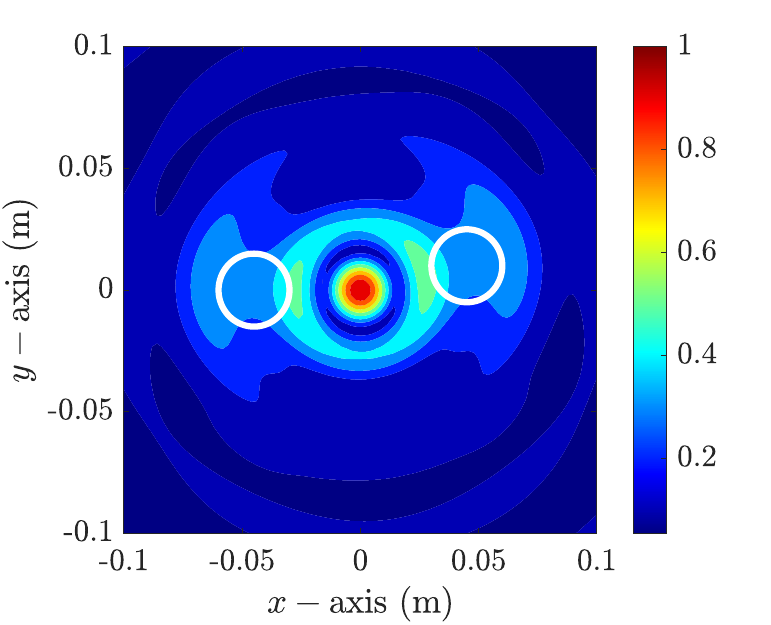}
\caption{\label{Fig5}(Example \ref{Ex5}) Maps of $\mathfrak{F}_{\msm}(\mr,C)$ with $\alpha=\SI{120}{\degree}$ and $C=0.1$ (left), $C=0.3i$ (middle), and $C=0.5$ (right).}
\end{figure}

\section{Concluding Remark}\label{sec:6}
In this paper, we have considered the application of the DSM to realize fast identification of small dielectric objects from a limited-aperture bistatic measurement dataset. To this end, indicator functions with single and multiple sources were designed by converting the unknown measurement data into a fixed constant. In addition, to explain the applicability of the indicator functions, as well as the influence of the converted constant and bistatic angle, we demonstrated that the indicator functions can be expressed by an infinite series of the Bessel functions of integer order, the material properties, and the applied constant and bistatic angle. Based on the theoretical results, we conclude that converting unknown measurement data into a zero constant and a small bistatic angle ensures good results. The results of various numerical simulations conducted on a 2D Fresnel dataset were presented to support these theoretical results.

\end{document}